\documentclass[review]{elsarticle}

\usepackage{lineno,hyperref}
\modulolinenumbers[5]

\journal{Linear Algebra and its Applications}

\begin{document}

\begin{frontmatter}

\title{On the positive stability of $P^2$-matrices}

\author{Olga Y. Kushel}
\address{Institut f\"{u}r Mathematik, MA 4-2, Technische Universit\"{a}t Berlin, Strasse des 17. Juni 136  D-10623 Berlin, Germany}
\ead{kushel@mail.ru}

\begin{abstract}
In this paper, we study the positive stability of $P$-matrices. We prove that a $P$-matrix ${\mathbf A}$ is positively stable if ${\mathbf A}$ is a $Q^2$-matrix and there is at least one nested sequence of principal submatrices of ${\mathbf A}$ each of which is also a $Q^2$-matrix. This result generalizes the result by Carlson which shows the positive stability of sign-symmetric $P$-matrices and the result by Tang, Simsek, Ozdaglar and Acemoglu which shows the positive stability of strictly row (column) square diagonally dominant for every order of minors $P$-matrices.
\end{abstract}

\begin{keyword}
$P$-matrices \sep $Q$-matrices \sep stable matrices \sep exterior products \sep $P$-matrix powers \sep stabilization
\MSC[2010] 15A48 \sep 15A18 \sep 15A75
\end{keyword}

\end{frontmatter}

\newtheorem{theorem}{Theorem}
\newtheorem{lemma}{Lemma}
\newtheorem{corollary}{Corollary}
\newproof{proof}{Proof}
\newproof{pmain}{Proof of Theorem \ref{maintheorem}}
\newproof{pcarl}{Proof of Theorem \ref{Carl}}
\newproof{ptang}{Proof of Theorem \ref{Tang}}
\newcommand{\diag}{{\rm diag}}
\linenumbers

\section{Introduction}
In this paper, we study $P$-matrices:

{\bf Definition.} A real $n \times n$ matrix $\mathbf A$ is called a {\it $P$-matrix} if all its principal minors are positive, i.e the inequality $A \left(\begin{array}{ccc}i_1 & \ldots & i_k \\ i_1 & \ldots & i_k \end{array}\right) > 0$
holds for all $(i_1, \ \ldots, \ i_k), \ 1 \leq i_1 < \ldots < i_k \leq n$, and all $k, \ 1 \leq k \leq n$.

Having been introduced by Fiedler and Ptak in the 1960s, $P$-matrices have found a great number of applications in various disciplines including physics, economics, communication networks and biology. Spectral properties of $P$-matrices have received great attention. The following result characterizes real eigenvalues of $P$-matrices (see \cite{FIP}, p. 385, Theorem 3.3).

\begin{theorem}[Fiedler, Pt\'{a}k]\label{FiP}
The following properties of a real matrix $\mathbf A$ are equivalent:
\begin{enumerate}
\item All principal minors of $\mathbf A$ are positive.
\item Every real eigenvalue of $\mathbf A$ as well as of each principal submatrix of $\mathbf A$ is positive.
\end{enumerate}
\end{theorem}

However, $P$-matrices may have non-real eigenvalues as well. Let us mention the following result by Kellogg concerning complex eigenvalues of $P$-matrices (see \cite{KEL}, p. 174, Corollary 1).
\begin{theorem}[Kellogg]\label{Kell}
Let $\lambda$ be an eigenvalue of an $n \times n$ $P$-matrix ${\mathbf A}$. Then
$$|{\rm arg}(\lambda)| < \pi - \frac{\pi}{n}.$$
\end{theorem}

The following generalization of the class of $P$-matrices preserves the above spectral properties (see \cite{HER2}, p. 83, Theorem 1).

{\bf Definition.} A real matrix $\mathbf A$ is called a {\it $Q$-matrix} if the inequality
$$\sum_{(i_1, \ldots, i_k)}A \left(\begin{array}{ccc}i_1 & \ldots & i_k \\ i_1 & \ldots & i_k \end{array}\right) > 0$$
holds for all $k, \ 1 \leq k \leq n$.

\begin{theorem}[Hershkowitz]\label{Her}
A set $\{\lambda_1, \ \ldots, \ \lambda_n\}$ of complex numbers is a spectrum of some $P$-matrix if and only if it is a spectrum of some $Q$-matrix.
\end{theorem}
\begin{corollary}\label{hersh}
Every real eigenvalue of a $Q$-matrix $\mathbf A$ is positive. Moreover, if $\lambda$ is an eigenvalue of an $n \times n$ $Q$-matrix then
$$|{\rm arg}(\lambda)| < \pi - \frac{\pi}{n}. $$
\end{corollary}

It is easy to see that a $P$-matrix may not be positively stable (a matrix is called {\it positively stable} if all of its eigenvalues have positive real parts). This paper deals with the relation between the positive stability and the positivity of principal minors of matrix powers. The major result on this topic was obtained by Carlson (see \cite{CARL2}, p. 1).

{\bf Definition.} A matrix $\mathbf A$ is called {\it sign-symmetric} if the inequality
$$A\left(\begin{array}{ccc}
  i_1 & \ldots & i_k \\
  j_1 & \ldots & j_k
\end{array}\right)A\left(\begin{array}{ccc}
  j_1 & \ldots & j_k \\
  i_1 & \ldots & i_k
\end{array}\right) \geq 0 $$
holds for all sets of indices $(i_1, \ \ldots, \ i_k), \ (j_1, \ \ldots, \ j_k)$, where $1 \leq i_1 < \ldots < i_k \leq n $, $1 \leq j_1 < \ldots < j_k \leq n $, $k = 1, \ \ldots, \ n$.

\begin{theorem}[Carlson]\label{Carl}
A sign-symmetric $P$-matrix is positively stable.
\end{theorem}

The proof of Carlson's theorem is based on the following implications.
\begin{enumerate}
\item[\rm (a)] $\mathbf A$ is a sign-symmetric $P$-matrix $\Rightarrow$ ${\mathbf A}^2$ is a $P$-matrix.
\item[\rm (b)] $\mathbf A$ is a sign-symmetric $P$-matrix $\Rightarrow$ ${\mathbf D}{\mathbf A}$ is a sign-symmetric $P$-matrix and $({\mathbf D}{\mathbf A})^2$ is a $P$-matrix for every diagonal matrix $\mathbf D$ with positive principal diagonal entries.
\end{enumerate}

Several attempts to generalize Carlson's theorem to wider classes of matrices were made afterwards. The following classes of matrices were introduced in \cite{HERK2}.

{\bf Definition.} A matrix $\mathbf A$ is called a {\it $P^2$-matrix ($Q^2$-matrix)} if both $\mathbf A$ and ${\mathbf A}^2$ are $P$- (respectively, $Q$-) matrices.

The following question was raised (see \cite{HERK2}, p. 122, question 6.2).

{\bf Are $P^2$-matrices positively stable?}

This question is still open. For the case of $Q^2$-matrices, Hershkowitz and Keller proved their positive stability for $n \leq 3$ (see \cite{HERK2}, p. 112, Proposition 3.1). However, the example given in \cite{HERJ1} (see \cite{HERJ1}, p. 164) and the reasoning in \cite{HERK2} (see \cite{HERK2}, p. 123, Corollary 6.9) show that $Q^2$-matrices of order $n \geq 4$ may have eigenvalues in the left-hand side of the complex plane. Some conditions sufficient for the positive stability were introduced in \cite{TSO}.

{\bf Definition.} An $n \times n$ matrix ${\mathbf A}$ is called strictly row square diagonally dominant for every order of minors if the following inequalities hold:
$$\left(A\left(\begin{array}{c} \alpha \\
\alpha \end{array}\right)\right)^2 > \sum_{\alpha,\beta \in [n], \alpha \neq \beta}\left( A\left(\begin{array}{c} \alpha \\
\beta \end{array}\right)\right)^2$$
for any $\alpha = (i_1, \ \ldots, \ i_k)$, $\beta = (j_1, \ \ldots, \ j_k)$ and all $k = 1, \ \ldots, \ n$.
A matrix $\mathbf A$ is called strictly column square diagonally dominant if ${\mathbf A}^T$ is strictly row square diagonally dominant.

It follows from the reasoning of \cite{TSO} that $P$-matrices which satisfy the above conditions are $Q^2$-matrices as well. The above conditions also guarantee the positive stability of a $P$-matrix (see \cite{TSO}, p. 27, Theorem 3).

\begin{theorem}[Tang et al.]\label{Tang} Let $\mathbf A$ be a $P$-matrix. If $\mathbf A$ is strictly row (column) diagonally dominant for every order of minors, then $\mathbf A$ is positively stable.
\end{theorem}

The main result of this paper generalizes both the results of \cite{CARL2} and \cite{TSO}. It also provides sufficient conditions for the stability of $P$-matrices which are also $Q^2$-matrices.

\begin{theorem} \label{maintheorem} Let an $n \times n$ $P$-matrix $\mathbf A$ also be a $Q^2$-matrix. Let $\mathbf A$ have a nested sequence of principal submatrices each of which is also a $Q^2$-matrix. Then $\mathbf A$ is positively stable.
\end{theorem}

Let us give an example illustrating Theorem \ref{maintheorem}.

{\bf Example.} Let $n = 4$ and
$${\mathbf A} = \left(\begin{array}{cccc} 6 & -30 & 1 & 1 \\ 1 & 2 & 1 & -5 \\ 1 & 1 & 10 & -10 \\ 1 & 1 & 1 & 10
\end{array}\right). $$

In this case, we have
$${\mathbf A}^{(2)} = \left(\begin{array}{cccccc} 42 & 5 & -31 & -32 & 148 & -6 \\
 36 & 59 & -61 & -301 & 299 & -20 \\
 36 & 5 & 59 & -31 & -301 & 9 \\
 -1 & 9 & -5 & 19 & -15 & 40 \\
 -1 & 0 & 15 & 1 & 25 & 15 \\
 0 & -9 & 20 & -9 & 20 & 110 \\
\end{array}\right);$$

$${\mathbf A}^{(3)} = \left(\begin{array}{cccc}
 383 & -241 & 254 & -1166 \\
 5 & 599 & 75 & -474 \\
 -324 & 720 & 631 & -3329 \\
 9 & -20 & 135 & 245 \\
\end{array}\right);$$

$$\det{\mathbf A} = 5491.$$

As we can see, $\mathbf A$ is a $P$-matrix. Since $\mathbf A$ is not sign-symmetric, it does not satisfy the conditions of Theorem \ref{Carl} (Carlson). Also, it does not satisfy the conditions of row (column) square diagonal dominance for every order of minors.

Let us check ${\mathbf A}^2$. We have
$${\mathbf A}^{2} = \left(\begin{array}{cccc}
 8 & -238 & -13 & 156 \\
 4 & -30 & 8 & -69 \\
 7 & -28 & 92 & -204 \\
 18 & -17 & 22 & 86 \\ \end{array}\right);$$
$$({\mathbf A}^{2})^{(2)} = \left(\begin{array}{cccccc}
712 & 116 & -1176 & -2294 & 21102 & -351 \\
1442 & 827 & -2724 & -22260 & 52920 & -11700 \\
4148 & 410 & -2120 & -5457 & -17816 & -4550 \\
98 & 312 & -333 & -2536 & 4188 & 4716 \\
472 & -56 & 1586 & -524 & -3753 & 2206 \\
385 & -1502 & 4274 & 948 & -5876 & 12400 \end{array}\right);$$
$$({\mathbf A}^{2})^{(3)} =  \left(\begin{array}{cccc}
52694 & -30462 & 82071 & -1463580 \\
-23656 & 421076 & 29530 & -655561 \\
-354897 & 1030264 & -79550 & -2879700 \\
-38188 & 78151 & 119046 & -390404 \\ \end{array}\right);$$
$${\rm Tr}({\mathbf A}^{2}) = 156 > 0;$$
$${\rm Tr}(({\mathbf A}^{2})^{(2)}) = 5530 > 0;$$
$${\rm Tr}(({\mathbf A}^{2})^{(3)}) = 3816 > 0;$$
$${\det}({\mathbf A}^2) = 30151081 > 0.$$
Thus ${\mathbf A}^{2}$ is a $Q^2$-matrix. Nevertheless, for the diagonal matrix
$${\mathbf D} = \diag\{1, \ 1, \ 0.1, \ 0.1\},$$
we have $${\mathbf D}{\mathbf A} = \left(\begin{array}{cccc} 6 & -30 & 1 & 1 \\ 1 & 2 & 1 & -5 \\ 0.1 & 0.1 & 1 & -1 \\ 0.1 & 0.1 & 0.1 & 1 \\ \end{array}\right);$$
$$({\mathbf D}{\mathbf A})^2 = \left(\begin{array}{cccc}
 6.2 & -239.8 & -22.9 & 156 \\
 7.6 & -26.4 & 3.5 & -15 \\
 0.7 & -2.8 & 1.1 & -2.4 \\
 0.81 & -2.69 & 0.4 & 0.5 \\
\end{array}\right)$$
and $${\rm Tr}(({\mathbf D}{\mathbf A})^2) = -18.6 < 0.$$
Thus $({\mathbf D}{\mathbf A})^2$ is not even a $Q$-matrix, Implication (b) does not hold and we can not apply the reasoning of the proof of Theorem \ref{Carl} (Carlson).

However, the matrix ${\mathbf A}$ satisfies the conditions of Theorem \ref{maintheorem}, since it has a nested sequence of principal submatrices, each of which is also a $Q^2$-matrix. We obtain this sequence by deleting the first, the second and the third row and column, consequently:
$${\mathbf A}_1 = \left(\begin{array}{cccc} 2 & 1 & -5 \\ 1 & 10 & -10 \\ 1 & 1 & 10 \\ \end{array}\right).$$
In this case
$${\mathbf A}_1^2 = \left(\begin{array}{cccc} 0 & 7 & -70 \\ 2 & 91 & -205 \\ 13 & 21 & 85 \\ \end{array}\right);$$
$$({\mathbf A}_1^2)^{(2)} = \left(\begin{array}{cccc} -14 & 140 & 4935 \\ -91 & 910 & 2065 \\ -1141 & 2835 & 12040 \\ \end{array}\right);$$
$$\det({\mathbf A}_1^2) = 60025 > 0;$$
$${\rm Tr}(({\mathbf A}_1^2)) = 176 > 0;$$
$${\rm Tr}(({\mathbf A}_1^2)^{(2)}) = 12936 > 0.$$
Thus ${\mathbf A}_1$ is a $Q^2$-matrix. Then we obtain
$${\mathbf A}_{12} = \left(\begin{array}{cc} 10 & -10 \\1 & 10 \end{array}\right)$$
with
$${\mathbf A}_{12}^2 = \left(\begin{array}{cc} 90 & -200 \\ 20 & 90 \\ \end{array}\right);$$
$$\det({\mathbf A}_{12}^2) = 12100 > 0;$$
and $${\rm Tr}(({\mathbf A}_1^2)^{(2)}) = 180 > 0.$$
Thus ${\mathbf A}_{12}$ is also a $Q^2$-matrix. For the matrix ${\mathbf A}_{123}$ which consists of only one positive entry $10$, the conditions are obvious.

It is not difficult to check that $\mathbf A$ is positively stable, with two pairs of complex adjoint eigenvalues $\lambda_{1,2} \approx 10.1979 \pm 2.0302i$ and $\lambda_{3,4} \approx 3.80215 \pm 6.02751i$.

This paper is organized as follows. In Section 2, we develop the necessary methods for the proof, namely, we describe the exterior products of operators and matrices and recall the definitions and statements concerning additive compound matrices introduced in \cite{FID}. In Section 3, the results on the stabilization by a diagonal matrix are studied. Recovering the proof from \cite{BAL}, we show the possibility of choosing a stabilization matrix with certain additional properties. Section 4 deals with the proof of the main result (Theorem \ref{maintheorem}). In Section 5, we analyze the known classes of positively stable $P^2$-matrices.
\section{Exterior products and additive compound matrices}
Let $\{e_1, \ \ldots, \ e_n\}$ be an arbitrary basis in $R^n$. Let $x_1, \ \ldots, \ x_j$ $ \ (2 \leq j \leq n)$ be any vectors in $R^n$ defined by their coordinates $x_i = (x_i^1, \ \ldots, \ x_i^n)$, $i = 1, \ \ldots, \ j$ in the basis $\{e_1, \ \ldots, \ e_n\}$. Then the {\it exterior product} $x_1 \wedge \ldots \wedge x_j$ of the vectors $x_1, \ \ldots, \ x_j$ is a vector in $ R^{n\choose j}$ (${n \choose j} = \frac{n!}{j!(n-j)!}$) with the coordinates of the form
$$(x_1 \wedge \ldots \wedge x_j)^{\alpha}:= \left|\begin{array}{ccc} x_1^{i_1} & \ldots & x_j^{i_1} \\
\ldots & \ldots & \ldots \\
x_1^{i_j} & \ldots & x_j^{i_j} \\
 \end{array}\right|,$$ where $\alpha$ is the number of the set of indices $(i_1, \ \ldots,  \ i_j) \subseteq [n]$ in the lexicographic ordering ( $[n]$, as usual, denotes the set $\{1, \ \ldots, \ n\}$).

We consider the $j$th exterior power $\wedge^j { R}^n$ of the space ${ R}^n$ as the space ${ R}^{n \choose j}$. The set of all exterior
products of the form $e_{i_1} \wedge \ldots \wedge e_{i_j}$, where
$1 \leq i_1 < \ldots < i_j \leq n$ forms a canonical basis in $\wedge^j { R}^n$ (see, for example, \cite{GLALU}).

Let us recall the following definition (see \cite{GROT}, p. 326).

{\bf Definition.} Given $j$ linear operators $A_1, \ \ldots, \ A_j$ on ${R}^n$, they define a linear operator $A_1 \wedge \ldots \wedge A_j$ on $\wedge^j
{R}^n$ by the following rule
$$ (A_1\wedge \ldots \wedge
A_j)(x_1 \wedge \ldots \wedge x_j) = \frac{1}{j!}\sum_\theta
A_{\theta(1)}x_1\wedge \ldots \wedge A_{\theta(j)}x_j ,$$
where the sum is taken with respect to all the permutations $\theta = (\theta(1), \ \ldots, \ \theta(j))$ of the set of indices $[j]$.
The operator $A_1 \wedge \ldots \wedge A_j$ is called an {\it exterior product} of the operators $A_1, \ \ldots, \ A_j$.

The above definition implies the property of commutativity: $$A_1 \wedge \ldots \wedge A_j = A_{\theta(1)} \wedge \ldots \wedge A_{\theta(j)}$$ for every permutation $\theta = (\theta(1), \ \ldots, \ \theta(j))$ of the set of indices $[j]$.

The following property also easily follows from this definition. Let $A_1, \ \ldots, \ A_j$ and $B_1, \ \ldots, \ B_j$ be two sets of linear operators on $R^n$. Then
$$(A_1\wedge \ldots \wedge
A_j) (B_1\wedge \ldots \wedge B_j) =
\frac{1}{j!}\sum_{\theta}(A_1B_{\theta(1)}\wedge
\ldots \wedge A_{j}B_{\theta(j)}).$$
In particular, $$(A_1\wedge \ldots \wedge
A_j)(B\wedge \ldots \wedge B) = A_1B \wedge \ldots \wedge A_jB,$$
for any linear operator $B$ on $R^n$.

Given $n \times n$ matrices ${\mathbf A}_1, \  \ldots, \ {\mathbf A}_j$, let us consider all of the possible
"mixed" minors of the $j$-th $(1 \leq j \leq n)$ order,
constructed of columns of different matrices. We denote such minors as follows: $$(A_{1}, \ldots, A_{j}) \left(\begin{array}{ccc} i_1 & \ldots & i_j \\
k_1^{\theta(1)} & \ldots & k_j^{\theta(j)}
\end{array}\right), $$
where $i_1, \ldots, i_j$ $(1 \leq i_1 < \ldots <
i_j \leq n)$ and $k_1,
 \ldots, k_j$ $(1 \leq k_1 < \ldots < k_j \leq n)$ are the numbers of rows and, respectively, columns, which form the minor. The notation $k_m^{\theta(m)}$ shows that the column with number $k_m$ belongs to the matrix ${\mathbf A}_{\theta(m)}$ (here $\theta = (\theta(1), \ \ldots, \ \theta(j))$ is an arbitrary permutation of the set of indices $[j]$).

{\bf Example.} Let the first matrix be ${\mathbf A} = \left(\begin{array}{cc} a_{11} & a_{12} \\
a_{21} & a_{22}
\end{array}\right)$ and the second matrix be $\mathbf B = \left(\begin{array}{cc} b_{11} & b_{12} \\
b_{21} & b_{22}
\end{array}\right).$ Let us take $(i_1, \ i_2) = (1, \ 2)$, $(k_1, \ k_2) = (1, \ 2)$. Then
$$(A, B)\left(\begin{array}{cc} 1 & 2 \\
1^2 & 2^1
\end{array}\right) = \left|\begin{array}{cc} b_{11} & a_{12} \\
b_{21} & a_{22}
\end{array}\right| \qquad \mbox{and} \qquad (A, B)\left(\begin{array}{cc} 1 & 2 \\
1^1 & 2^2
\end{array}\right) = \left|\begin{array}{cc} a_{11} & b_{12} \\
a_{21} & b_{22}
\end{array}\right|.$$

{\bf Definition.} An {\it exterior product} of the matrices ${\mathbf A}_1, \ \ldots, \ {\mathbf A}_j$ is an ${n \choose j} \times {n \choose j}$ matrix ${\mathbf A}_1 \wedge
\ldots \wedge {\mathbf A}_j$ with the entries $\zeta_{\alpha\beta},$ $1 \leq \alpha, \beta \leq {n \choose j}$ of the following form: $$ \zeta_{\alpha\beta} = \frac{1}{j!}\sum\limits_{\theta}(A_{1}, \ldots, A_{j}) \left(\begin{array}{ccc} i_1 & \ldots & i_j \\
k_1^{\theta(1)} & \ldots & k_j^{\theta(j)} \\
\end{array}\right),$$ where $\alpha, \ \beta$ are the numbers in the lexicographic numeration of the sets of indices $(i_1, \ \ldots, \ i_j)$ and $(k_1, \ \ldots, \ k_j)$, respectively.

 {\bf Example.} Let $${\mathbf A} =
\left(\begin{array}{ccc}
 a_{11} & a_{12} & a_{13} \\
 a_{21} & a_{22} & a_{23} \\
 a_{31} & a_{32} & a_{33}  \\
\end{array}\right), \qquad {\mathbf B} =
\left(\begin{array}{ccc}
 b_{11} & b_{12} & b_{13}  \\
 b_{21} & b_{22} & b_{23} \\
 b_{31} & b_{32} & b_{33}  \\
\end{array}\right).$$ Then
$${\mathbf A}\wedge {\mathbf B} = \left(\begin{array}{ccc}
  \zeta_{11} & \zeta_{12} & \zeta_{13} \\
  \zeta_{21} & \zeta_{22} & \zeta_{23} \\
  \zeta_{31} & \zeta_{32} & \zeta_{33}
\end{array}\right),$$
where $$ \zeta_{11} = \frac{1}{2}\left(\left|\begin{array}{cc}a_{11} &
b_{12} \\ a_{21} & b_{22} \\ \end{array}\right| +
 \left|\begin{array}{cc} b_{11} & a_{12} \\ b_{21} & a_{22} \\ \end{array}\right|\right);$$
$$ \zeta_{12} = \frac{1}{2}\left(
\left|\begin{array}{cc} a_{11} &
b_{13} \\ a_{21} & b_{23} \\ \end{array}\right| +
   \left|\begin{array}{cc} b_{11} & a_{13} \\ b_{21} & a_{23} \\ \end{array}\right|\right);$$
$$ \zeta_{13} = \frac{1}{2}\left( \left|\begin{array}{cc}a_{12} &
b_{13} \\ a_{22} & b_{23} \\ \end{array}\right|
 + \left|\begin{array}{cc} b_{12} & a_{13} \\ b_{22} & a_{23} \\ \end{array}\right|\right);$$
$$ \zeta_{21} = \frac{1}{2}\left(\left|\begin{array}{cc} a_{11} &
b_{12} \\ a_{31} & b_{32} \\ \end{array}\right|
 + \left|\begin{array}{cc} b_{11} & a_{12} \\ b_{31} & a_{32} \\ \end{array}\right| \right);$$
$$ \zeta_{22} = \frac{1}{2}\left( \left|\begin{array}{cc} a_{11} &
b_{13} \\ a_{31} & b_{33} \\ \end{array}\right|
  + \left|\begin{array}{cc}b_{11} & a_{13} \\b_{31} & a_{33} \\ \end{array}\right|\right);$$
$$ \zeta_{23} = \frac{1}{2}\left(\left|\begin{array}{cc}a_{12} &
b_{13} \\ a_{32} & b_{33} \\ \end{array}\right|
 + \left|\begin{array}{cc}b_{12} & a_{13} \\ b_{32} & a_{33} \\ \end{array}\right| \right);$$
$$ \zeta_{31} = \frac{1}{2}\left(\left|\begin{array}{cc}a_{21} &
b_{22} \\ a_{31} & b_{32} \\ \end{array}\right|
   + \left|\begin{array}{cc} b_{21} & a_{22} \\ b_{31} & a_{32} \\ \end{array}\right|\right);$$
$$ \zeta_{32} = \frac{1}{2}\left( \left|\begin{array}{cc} a_{21} &
b_{23} \\ a_{31} & b_{33} \\ \end{array}\right|
  + \left|\begin{array}{cc} b_{21} & a_{23} \\ b_{31} & a_{33} \\ \end{array}\right|\right);$$
$$ \zeta_{33} = \frac{1}{2}\left(\left|\begin{array}{cc} a_{22} &
b_{23} \\ a_{32} & b_{33} \\ \end{array}\right|
 + \left|\begin{array}{cc} b_{22} & a_{23} \\ b_{32} & a_{33} \\ \end{array}\right|\right).$$

Let $j$ linear operators $A_1, \ \ldots, \ A_j: R^n \rightarrow R^n$ be given by their matrices ${\mathbf A}_1, \ \ldots, \ {\mathbf A}_j$, respectively, in the basis $\{e_1, \ \ldots, \ e_n\}$. It is easy to see that the matrix of the operator $A_1 \wedge \ldots \wedge A_j$ in the basis
$\{e_{i_1} \wedge \ldots \wedge e_{i_j}\}$, where $1 \leq i_1<
\ldots < i_j \leq n$, equals the exterior product ${\mathbf A}_1 \wedge
\ldots \wedge {\mathbf A}_{j}$ of the matrices ${\mathbf A}_1, \ \ldots, \ {\mathbf A}_j$.

The following special case of the exterior products of operators is studied in \cite{FID} (see \cite{FID}, p. 394).

{\bf Definition.} For a linear operator $A: R^n \rightarrow R^n$ and two positive integers $j, \ m$ $(1 \leq m \leq j \leq n)$, the {\it generalized $j$-th compound
operator} $\wedge_m^j A$ is defined by the following formula:
$$\wedge_m^j A = \underbrace{A \wedge \ldots \wedge A}_{m} \wedge
\underbrace{I \wedge \ldots \wedge I}_{j-m}.$$

The matrix of $\wedge_m^j A$ in the canonical basis $\{e_{i_1} \wedge \ldots \wedge e_{i_j}\}$, where $1 \leq i_1< \ldots < i_j \leq n$ is called the {\it generalized $j$th compound matrix} of the initial matrix $\mathbf A$ and denoted ${\mathbf A}_{m}^{(j)}$. It follows from the above definitions that ${\mathbf A}_{m}^{(j)} = \underbrace{{\mathbf A} \wedge \ldots \wedge {\mathbf A}}_{m} \wedge
\underbrace{{\mathbf I} \wedge \ldots \wedge {\mathbf I}}_{j-m}$.

{\bf Example.}
Let us consider an $n \times n$ diagonal matrix $\mathbf D$ of the form:
$${\mathbf D} = \diag\{\epsilon_1, \ \epsilon_2, \ \ldots, \ \epsilon_n\},$$
where $1 = \epsilon_1 > \epsilon_2 > \ldots > \epsilon_n > 0$.
In this case, ${\mathbf D}_m^{(j)} = \underbrace{{\mathbf D} \wedge \ldots \wedge {\mathbf D}}_{m} \wedge \underbrace{{\mathbf I} \wedge \ldots \wedge \mathbf{I}}_{j-m}$, where $1 \leq m \leq j \leq n$, is an ${n \choose j}\times {n \choose j}$ diagonal matrix of the form:
\begin{equation}\label{DiagForm}
{\mathbf D}_m^{(j)} = \diag\{d^m_{11}, \ \ldots, \ d^m_{{n \choose j}{n \choose j}}\},
\end{equation}
where $$d^m_{\alpha\alpha} = \sum_{(k_1, \ldots, k_m) \subseteq (i_1, \ \ldots, \ i_j)}\left|\begin{array}{cccc} \epsilon_{k_1} & 0 &  \ldots & 0 \\
  0 & \epsilon_{k_2} & \ldots & 0 \\
  \ldots &  \ldots & \ldots & \ldots \\
  0 & 0 & \ldots & \epsilon_{k_m} \end{array}\right| = \sum_{(k_1, \ldots, k_m) \subseteq (i_1, \ \ldots, \ i_j)}\epsilon_{k_1}\ldots \epsilon_{k_m},$$
$\alpha$ is the number in the lexicographic numeration of the set of indices $(i_1, \ \ldots, \ i_j)$, $1 \leq i_1 \leq \ldots \leq i_j \leq n$, the sum is taken with respect to all the possible subsets of $m$ indices $(k_1, \ldots, k_m)$ from the set $(i_1, \ \ldots, \ i_j)$, $i_1 \leq k_1 < \ldots < k_m \leq i_j$.
For example, in the case of $n = 3$, we have $${\mathbf D} =
\left(\begin{array}{ccc}
 \epsilon_1 & 0 & 0 \\
 0 & \epsilon_2 & 0 \\
 0 & 0 & \epsilon_3  \\
\end{array}\right), \qquad {\mathbf I} =
\left(\begin{array}{ccc}
 1 & 0 & 0  \\
 0 & 1 & 0 \\
 0 & 0 & 1  \\
\end{array}\right).$$ In this case
$${\mathbf D}_1^{(2)} = \left(\begin{array}{ccc}
  \epsilon_1 + \epsilon_2 & 0 & 0 \\
  0 & \epsilon_1 + \epsilon_3 & 0 \\
  0 & 0 & \epsilon_2 + \epsilon_3
\end{array}\right).$$

In the case, when $m = j$, the above definition gives the linear operator $\wedge^j A$, defined by the equality
$$ (\wedge^j A)(x_1 \wedge \ldots \wedge x_j) = Ax_1 \wedge \ldots \wedge Ax_j.$$
The operator $\wedge^j A$ is called the {\it $j$th exterior power} of the initial operator $A$ or the {\it jth compound operator}. It is easy to see that $\wedge^1 A = A$ and $\wedge ^n A$ is one-dimensional and coincides with $\det A$.

If ${\mathbf A} =
\{a_{ij}\}_{i,j = 1}^n$ is the matrix of $A$ in the basis $\{e_1, \ \ldots, \ e_n\}$, then
the matrix of $\wedge^j A$ in the basis
$\{e_{i_1} \wedge \ldots \wedge e_{i_j}\}$, where $1 \leq i_1<
\ldots < i_j \leq n$, equals the $j$th compound matrix $
{\mathbf A}^{(j)}$ of the initial matrix ${\mathbf A}$. (Here the {\it $j$th compound matrix $
{\mathbf A}^{(j)}$} consists of all the minors of the $j$th order
$A\left(\begin{array}{ccc}
  i_1 &  \ldots & i_j \\
  k_1 & \ldots & k_j \end{array}\right)$, where $1 \leq i_1<
\ldots < i_j \leq n, \ 1 \leq k_1< \ldots < k_j \leq n$, of the
initial $n \times n$ matrix ${\mathbf A}$, listed in the lexicographic order (see, for example,
\cite{PINK})).

The following properties of compound matrices are well-known.
\begin{enumerate}
\item[\rm 1.] Let ${\mathbf A}, \ {\mathbf B}$ be $n \times n$ matrices. Then $({\mathbf A}{\mathbf B})^{(j)} = {\mathbf A}^{(j)}{\mathbf B}^{(j)}$ for $j = 1, \ \ldots, \ n$ (the Cauchy--Binet formula).
\item[\rm 2.] The $j$-th compound matrix ${\mathbf A}^{(j)}$ of an invertible matrix $\mathbf A$ is also invertible and the following equality holds: $({\mathbf A}^{(j)})^{-1} = ({\mathbf A}^{-1})^{(j)}$, $j = 1, \ \ldots, \ n$ (the Jacobi formula).
\end{enumerate}

\section{Stabilization by a diagonal matrix}
Here we will use the following definitions and notations (see, for example, \cite{GRU}).

{\bf Definition.} We say that an $n \times n$ matrix $\mathbf A$ has a {\it nested sequence of positive principal minors} or simply a {\it nest}, if there is a permutation $(i_1, \ \ldots, \ i_n)$ of the set of indices $[n]$ such that
$$A\left(\begin{array}{ccc}i_1 & \ldots  & i_j \\ i_1 & \ldots & i_j \end{array}\right) > 0 \qquad j = 1, \ \ldots, \ n.$$
Note that the indices $(i_1, \ \ldots, \ i_j)$ in each minor are ordered lexicographically.

The nest, defined by the natural ordering $\{1, \ \ldots, \ n\}$ is called the {\it leading nest}.

Later we'll use "a positive diagonal matrix" for a diagonal matrix with positive principal diagonal entries.
To describe the matrix stablization, let us give the following definitions.

{\bf Definition.} An $n \times n$ positive diagonal matrix ${\mathbf D}_{(\mathbf A)}$ is called a {\it stabilization matrix} for an $n \times n$ matrix $\mathbf A$ if all the eigenvalues of ${\mathbf D}_{(\mathbf A)}{\mathbf A}$ are positive and simple.

{\bf Definition.} An $n \times n$ real matrix $\mathbf A$ is called {\it stabilizable} if there is at least one stabilization matrix ${\mathbf D}_{(\mathbf A)}$.

The following sufficient conditions for the existence of a stabilization matrix were provided by Fisher and Fuller (see \cite{BAL}, p. 728, Theorem 1, also \cite{FIF}).
\begin{theorem}[Fisher, Fuller]\label{FiF} Let $\mathbf A$ be an $n \times n$ real matrix, all of whose leading principal minors are positive. Then there is an $n \times n$ positive diagonal matrix ${\mathbf D}_{({\mathbf A})}$, such that all of the eigenvalues of ${\mathbf D}_{({\mathbf A})}{\mathbf A}$ are positive and simple.
\end{theorem}

An obvious consiquence of the Fisher--Fuller theorem is the following statement.

\begin{corollary}\label{Fifc}
An $n \times n$ real matrix $\mathbf A$ is stabilizable if it has at least one nested sequence of positive principal minors.
\end{corollary}

Let us prove the following lemma which describes the possible choice of a stabilization matrix.

\begin{lemma}\label{diag}
Let $\mathbf A$ be an $n \times n$ matrix with positive leading principal minors. Then it is stabilizable and the following statements hold:
\begin{enumerate}
\item[\rm 1.] We can choose the stabilization matrix ${\mathbf D}_{({\mathbf A})}$ in the following form
    $${\mathbf D}_{({\mathbf A})} = \diag\{\epsilon_1, \ \epsilon_2, \ \ldots, \ \epsilon_{n}\},$$
    where $1 = \epsilon_1 > \epsilon_2 > \ldots > \epsilon_n > 0.$
\item[\rm 2.] There is a stabilization matrix $${\mathbf D}^0_{({\mathbf A})} = \diag\{\epsilon^0_1, \ \epsilon^0_2, \ \ldots, \ \epsilon^0_{n}\},$$ such that any positive diagonal matrix ${\mathbf D}_{({\mathbf A})} = \diag\{\epsilon_1, \ \epsilon_2, \ \ldots, \ \epsilon_{n}\}$ satisfying $\epsilon_1 = \epsilon_1^0$ and
    $$\frac{\epsilon_i}{\epsilon_{i+1}} \geq \frac{\epsilon_i^0}{\epsilon_{i+1}^0} \qquad i = 1, \ \ldots, \ n-1. $$ is also a stabilization matrix for $\mathbf A$.
\end{enumerate}
\end{lemma}
\begin{proof} For the proof, we repeat the reasoning of Proof 1 of the Fisher--Fuller theorem (see \cite{BAL}, p. 728-729). We use the induction on $n$. For $n = 1$, the result is trivial, ${\mathbf A} = \{a_{11}\}$ and ${\mathbf D}_{({\mathbf A})} = \{1\}$. Let us prove Lemma \ref{diag} for $n =2$. In this case, ${\mathbf A} = \left(\begin{array}{ccc} a_{11} & a_{12} \\ a_{21} & a_{22} \end{array}\right)$ and we define ${\mathbf D}_{({\mathbf A})} = \left(\begin{array}{ccc} 1 & 0 \\ 0 & \epsilon_2 \end{array}\right)$, where $\epsilon_2 > 0$ will be chosen later. Then ${\mathbf D}_{({\mathbf A})}{\mathbf A} = \left(\begin{array}{ccc} a_{11} & a_{12} \\ \epsilon_2 a_{21} & \epsilon_2 a_{22} \end{array}\right)$. Let $M_2(\epsilon_2, \lambda) := \det({\mathbf D}_{({\mathbf A})}{\mathbf A} - \lambda I)$, which is dependent on $\epsilon_2$.
Then $M_2(0, \lambda)$ has two simple roots: $\lambda_1 = a_{11}$ and $\lambda_2 = 0$. The roots of $M_2(\epsilon_2, \lambda)$ are close to $a_{11}$ and $0$, for sufficiently small values of $\epsilon_2$. That means, both of them are real (since complex eigenvalues must appear in conjugate pairs), and at least one of them is positive.
Since $\det({\mathbf D}_{({\mathbf A})}{\mathbf A})$ is positive, the other eigenvalue is also positive. The continuity of eigenvalues implies, that there is a positive integer $\epsilon_2^0 < 1$ such that the above spectral properties hold for all $\epsilon_2,$
$0 < \epsilon_2 \leq \epsilon_2^0$. So we put ${\mathbf D}_{({\mathbf A})}^0:= \diag\{1, \ \epsilon_2^0\}$. The above reasoning shows that any positive diagonal matrix ${\mathbf D}_{({\mathbf A})}:= \diag\{1, \ \epsilon_2\}$ with $0 < \epsilon_2 \leq \epsilon_2^0$ is also a stabilization matrix for ${\mathbf A}$.

Note that even in the case $n=2$, if ${\mathbf D}_{({\mathbf A})} = \{1, \ \epsilon_2\}$ is an arbitrary stabilization matrix, not every matrix of the form $\widetilde{{\mathbf D}}_{({\mathbf A})} = \diag\{1, \widetilde{\epsilon}_2\}$, where $0 < \widetilde{\epsilon}_2 < \epsilon_2$ will be also a stabilization matrix. This is true only for ${\mathbf D}_{({\mathbf A})}^0$.

For $n=2$, the Lemma is proven. Assume the Lemma holds for $n-1$. Let us prove it for $n$. Let ${\mathbf A}$ be an $n \times n$ matrix and let us show that $\epsilon_{n}$ can be chosen to satisfy the inequalities $ 0 < \epsilon_{n} < \epsilon_{n-1}$.
Repeating the proof of the Fisher--Fuller theorem, we apply the induction hypothesis to the $(n-1) \times (n-1)$ leading principal submatrix ${\mathbf A}_{11}$, obtained from $\mathbf A$ by deleting the last row and the last column. We use the following partitions:
$${\mathbf A} = \left(\begin{array}{ccc} {\mathbf A}_{11} & {\mathbf A}_{12} \\ {\mathbf A}_{21} & {\mathbf A}_{22} \end{array}\right),$$
and
$${\mathbf D}_{({\mathbf A})} = \left(\begin{array}{ccc} {\mathbf D}_1 & 0 \\ 0 & \epsilon_{n} \end{array}\right),$$
where ${\mathbf D}_1 = {\rm diag} \{\epsilon_1, \ \epsilon_2, \ \ldots, \ \epsilon_{n-1}\}$, with $1 = \epsilon_1 > \epsilon_2 > \ldots > \epsilon_{n-1} > 0$ and $\epsilon_n$ will be chosen later.

We also represent the matrix ${\mathbf D}^0_{({\mathbf A})}$ as
$${\mathbf D}^0_{({\mathbf A})} = \left(\begin{array}{ccc} {\mathbf D}^0_1 & 0 \\ 0 & \epsilon^0_{n} \end{array}\right),$$
where the matrix ${\mathbf D}^0_1$ is obtained by applying the induction hypothesis to the matrix ${\mathbf A}_{11}$ and $\epsilon^0_{n}$ will be chosen later.

Repeating the above reasoning, we put $M_{n}(\epsilon_n, \lambda) := \det({\mathbf D}_{({\mathbf A})}{\mathbf A} - \lambda I)$. Nonzero roots of $M_{n}(0, \lambda)$ are the same that of $\det({\mathbf D}_1{\mathbf A}_{11} - \lambda I)$, besides, it has a simple root in $0$. Considering sufficiently small $\epsilon_{n}$ and taking into account that $\det({\mathbf D}_{({\mathbf A})}{\mathbf A}) > 0$, we obtain that all the roots of $M_{n}(\epsilon_{n}, \lambda)$ are positive and simple. Using the continuity of eigenvalues, we get that there is a positive integer $\epsilon_n^0$, $0 < \epsilon_{n}^0 < \epsilon_{n-1}^0$ such that the above spectral properties hold for all $\epsilon_n$ which satisfy $0 < \epsilon_{n} \leq \epsilon_{n}^0$. Thus we can choose $\epsilon_n$ satisfying $0 < \epsilon_{n} < \min(\epsilon_n^0, \epsilon_{n-1})$. So we have found the stabilization matrix ${\mathbf D}_{({\mathbf A})}$ in the form ${\mathbf D}_{({\mathbf A})} = \diag\{\epsilon_1, \ \epsilon_2, \ \ldots, \ \epsilon_n\}$, where $1 = \epsilon_1 > \epsilon_2 > \ldots > \epsilon_n > 0$.

Now we put
$${\mathbf D}^0_{({\mathbf A})} = \left(\begin{array}{ccc}{\mathbf D}^0_1 & 0 \\ 0 & \epsilon^0_{n} \end{array}\right),$$
where $\epsilon_n^0$ is as above. Let us prove that any positive diagonal matrix   ${\mathbf D}_{({\mathbf A})} = \diag\{\epsilon_1, \ \epsilon_2, \ \ldots, \ \epsilon_{n}\}$ satisfying $\epsilon_1 = 1$ and
    $$\frac{\epsilon_i}{\epsilon_{i+1}} \geq \frac{\epsilon_i^0}{\epsilon_{i+1}^0} \qquad i = 1, \ \ldots, \ n-1. $$ is also a stabilization matrix for $\mathbf A$.
Let us write ${\mathbf D}_{({\mathbf A})}$ in the partition form:
$${\mathbf D}_{({\mathbf A})} = \left(\begin{array}{ccc}{\mathbf D}_1 & 0 \\ 0 & \epsilon_{n} \end{array}\right).$$
By the induction hypothesis, ${\mathbf D}_1$ stabilizes ${\mathbf A}_{11}$. Then the previous part of the proof implies that it is enough to show that the last entry $\epsilon_n$ satisfies the inequality $\epsilon_n \leq \epsilon_n^0$. Indeed, multiplying the inequalities
$$\frac{\epsilon_i}{\epsilon_{i+1}} \geq \frac{\epsilon_i^0}{\epsilon_{i+1}^0} \qquad i = 1, \ \ldots, \ n-1. $$
we obtain that $\frac{\epsilon_1}{\epsilon_n} \geq \frac{\epsilon_1^0}{\epsilon_n^0}$ and taking into account that $\epsilon_1 = \epsilon_1^0 = 1$, we obtain that $0 < \epsilon_n \leq \epsilon_n^0$. Thus ${\mathbf D}_{({\mathbf A})}$ is also a stabilization matrix for ${\mathbf A}$. \end{proof}

\section{Spectra of $P^2$-matrices}

The following properties of $P$-matrices are well-known (see, for example, \cite{TSA}, Theorem 3.1).

\begin{lemma}\label{P} Let $\mathbf A$ be a $P$-matrix. Then the following matrices are also $P$-matrices.
\begin{enumerate}
\item[\rm 1.] ${\mathbf A}^T$ (the transpose of $\mathbf A$);
\item[\rm 2.] ${\mathbf A}^{-1}$ (the inverse of $\mathbf A$);
\item[\rm 3.] ${\mathbf D}{\mathbf A}{\mathbf D}^{-1}$, where $\mathbf D$ is an invertible diagonal matrix;
\item[\rm 4.] ${\mathbf P}{\mathbf A}{\mathbf P}^{-1}$, where $\mathbf P$ is a permutation matrix;
\item[\rm 5.] any principal submatrix of $\mathbf A$;
\item[\rm 6.] the Schur complement of any principal submatrix of $\mathbf A$.
\item[\rm 7.] both ${\mathbf D}{\mathbf A}$ and ${\mathbf A}{\mathbf D}$, where $\mathbf D$ is a positive diagonal matrix.
\end{enumerate}
\end{lemma}

Let us list the following properties of $P^2$- and $Q^2$- matrices, which will be used later.

\begin{lemma}\label{P^2}
Let $\mathbf A$ be a $P^2$- ($Q^2$-) matrix. Then the following matrices are also $P^2$- ($Q^2$-) matrices.
\begin{enumerate}
\item[\rm 1.] ${\mathbf A}^T$ (the transpose of $\mathbf A$);
\item[\rm 2.] ${\mathbf A}^{-1}$ (the inverse of $\mathbf A$);
\item[\rm 3.] ${\mathbf D}{\mathbf A}{\mathbf D}^{-1}$, where $\mathbf D$ is an invertible diagonal matrix;
\item[\rm 4.] ${\mathbf P}{\mathbf A}{\mathbf P}^{-1}$, where $\mathbf P$ is a permutation matrix.
\end{enumerate}
\end{lemma}
\begin{proof} (1) Since $\mathbf A$ is a $P$-matrix, ${\mathbf A}^T$ is also a $P$-matrix, by Lemma \ref{P}, Part 1. Since ${\mathbf A}^2$ is a $P$-matrix and $({\mathbf A}^2)^T = ({\mathbf A}^T)^2$, we obtain that $({\mathbf A}^T)^2$ is also a $P$-matrix. So we conclude that ${\mathbf A}^T$ is a $P^2$-matrix.

(2) It is enough for the proof just to check that $({\mathbf A}^{-1})^2 = ({\mathbf A}^2)^{-1}$ and to apply Lemma \ref{P}, Part 2 to both $\mathbf A$ and ${\mathbf A}^2$.

(3) The proof follows from the equality $({\mathbf D}{\mathbf A}{\mathbf D}^{-1})^2 = {\mathbf D}({\mathbf A})^2{\mathbf D}^{-1}$ and Part 3 of Lemma \ref{P}.

(4) The proof follows from the equality $({\mathbf P}{\mathbf A}{\mathbf P}^{-1})^2 = {\mathbf P}({\mathbf A})^2{\mathbf P}^{-1}$ and Part 4 of Lemma \ref{P}.

The case of a $Q^2$-matrix is considered analogically by using the obvious fact that if $\mathbf A$ is a $Q$-matrix then ${\mathbf A}^T$, ${\mathbf A}^{-1}$, ${\mathbf D}{\mathbf A}{\mathbf D}^{-1}$, ${\mathbf P}{\mathbf A}{\mathbf P}^{-1}$ are also $Q$-matrices.
\end{proof}

Note, that principal submatrices and Schur complements of a $P^2$-matrix may not be $P^2$-matrices. Unlike the case of $P$-matrices, the product of the form ${\mathbf D}{\mathbf A}$, where $\mathbf A$ is a $P^2$-matrix, $\mathbf D$ is an arbitrary positive diagonal matrix, also may not be a $P^2$-matrix and even a $Q^2$-matrix.

The following lemma shows the link between the properties of the compound matrices of diagonal scalings and generalized compound matrices.

\begin{lemma}\label{tech}
Let an $n \times n$ $P$-matrix $\mathbf A$ be also a $Q^2$-matrix. Let $\mathbf D$ be a positive diagonal matrix such that the following inequalities hold:
$${\rm Tr}({\mathbf D}_k^{(j)}{\mathbf A}^{(j)}{\mathbf D}_m^{(j)}{\mathbf A}^{(j)}) > 0$$ for every $j = 1, \ \ldots, \ n$ and every pair $(k,m)$, $1 \leq k,m \leq j$.

Then ${\mathbf D}_t{\mathbf A} = (t{\mathbf I} + (1-t){\mathbf D}){\mathbf A}$ is also a $P$-matrix and a $Q^2$-matrix for all $t \in [0,1]$.
\end{lemma}
\begin{proof} To prove that ${\mathbf D}_t{\mathbf A}$ is a $P$-matrix, it is enough to observe that ${\mathbf D}_t$ is a positive diagonal matrix for all $t \in [0,1]$ and then apply Lemma \ref{P}, Part 7. Now let us prove that $({\mathbf D}_t{\mathbf A})^2$ is a $Q$-matrix, or, equivalently, that ${\rm Tr}((({\mathbf D}_t{\mathbf A})^2)^{(j)}) > 0$ for all $j = 1, \ \ldots, \ n$ and all $t \in [0,1]$.

Applying the Cauchy--Binet formula, we obtain that
$$(({\mathbf D}_t{\mathbf A})^2)^{(j)} = (({\mathbf D}_t{\mathbf A})^{(j)})^2,$$ for all $j = 1, \ \ldots, \ n$ and all $t \in [0,1]$.

Using the Cauchy--Binet formula again, we expand $({\mathbf D}_t{\mathbf A})^{(j)}$ as follows:
$$({\mathbf D}_t{\mathbf A})^{(j)} = ({\mathbf D}_t)^{(j)}{\mathbf A}^{(j)} = (t{\mathbf I} + (1-t){\mathbf D})^{(j)}{\mathbf A}^{(j)}.$$

Now, using the properties of the exterior products and the binomial formula, we obtain
$$(t{\mathbf I} + (1-t){\mathbf D})^{(j)} = \sum_{k = 0}^j{j \choose k}t^{j-k}(1-t)^k\underbrace{{\mathbf I} \wedge \ldots \wedge {\mathbf I}}_{j-k}\wedge \underbrace{{\mathbf D} \wedge \ldots \wedge {\mathbf D}}_k =$$
$$ = \sum_{k = 0}^j{j \choose k}t^{j-k}(1-t)^k {\mathbf D}_k^{(j)},$$
where ${\mathbf D}_0^{(j)} = {\mathbf I}^{(j)}$.
Thus we have
$$((t{\mathbf I} + (1-t){\mathbf D})^{(j)}{\mathbf A}^{(j)})^2 = $$ $$ = \left(\sum_{k = 0}^j{j \choose k}t^{j-k}(1-t)^k {\mathbf D}_k^{(j)}{\mathbf A}^{(j)}\right)\left(\sum_{m = 0}^j{j \choose m}t^{j-m}(1-t)^m {\mathbf D}_m^{(j)}{\mathbf A}^{(j)}\right) =$$
$$ = \sum_{k,m = 0}^j{j \choose k}{j \choose m}t^{2j-(k+m)}(1-t)^{k+m} {\mathbf D}_k^{(j)}{\mathbf A}^{(j)}{\mathbf D}_m^{(j)}{\mathbf A}^{(j)}.$$

Since the trace function is linear, we obtain
$${\rm Tr}(((t{\mathbf I} + (1-t){\mathbf D})^{(j)}{\mathbf A})^{(j)})^2) = $$
\begin{equation}\label{Sum}
 \sum_{k,m = 0}^j{j \choose k}{j \choose m}t^{2j-(k+m)}(1-t)^{k+m} {\rm Tr}({\mathbf D}_k^{(j)}{\mathbf A}^{(j)}{\mathbf D}_m^{(j)}{\mathbf A}^{(j)}).
\end{equation}
According to the conditions of the Lemma, each term of Sum \ref{Sum} is positive. Thus the whole sum is positive.
 \end{proof}

Let us observe the following notations and basic facts.
Let ${\mathbf A} = \{a_{ij}\}_{i,j = 1}^n$ be an $n \times n$ matrix, then $|{\mathbf A}|$ denotes a matrix which entries are equal to the absolute values of $a_{ij}$ $(i, j = 1, \ \ldots, \ n)$, i.e. $|{\mathbf A}| = \{|a_{ij}|\}_{i,j = 1}^n$.

In this case, the triangle inequality shows that $|{\mathbf A}{\mathbf B}| \leq |{\mathbf A}||{\mathbf B}|$ entry-wise for any $n \times n$ matrices $\mathbf A$ and $\mathbf B$.

Let ${\mathbf D} = \diag\{d_{11}, \ \ldots, \ d_{nn}\}$ be a positive diagonal matrix which satisfies the estimate ${\mathbf D} < \alpha{\mathbf I}$ for some positive value $\alpha$ (in this notation, we mean that $d_{ii} < \alpha$ for all $i = 1, \ \ldots, \ n$). Then
\begin{equation}\label{oTrace}
\left|{\rm Tr}({\mathbf D}{\mathbf A})\right| = \left|\sum_{i=1}^n d_{ii}a_{ii}\right| \leq \sum_{i=1}^n d_{ii}|a_{ii}| < \alpha \left(\sum_{i=1}^n|a_{ii}|\right) = \alpha{\rm Tr}|{\mathbf A}|,
\end{equation}
for any $n \times n$ matrix $\mathbf A$.

Let ${\mathbf D}_1$, ${\mathbf D}_2$ be positive diagonal matrices which satisfy the estimates ${\mathbf D}_1 < \alpha{\mathbf I}$, ${\mathbf D}_2 < \beta{\mathbf I}$, respectively, for some positive values $\alpha$ and $\beta$. Then:
\begin{equation}\label{ooTrace}
\left|{\rm Tr}({\mathbf D}_1{\mathbf A}{\mathbf D}_2{\mathbf A})\right| < \alpha{\rm Tr}(|{\mathbf A}{\mathbf D}_2{\mathbf A}|) \leq \alpha{\rm Tr}(|{\mathbf A}||{\mathbf D}_2||{\mathbf A}|) = \end{equation}
$$ \alpha{\rm Tr}({\mathbf D}_2|{\mathbf A}|^2) < \alpha\beta{\rm Tr}(|{\mathbf A}|^2),$$
for any $n \times n$ matrix $\mathbf A$.

Given an $n \times n$ matrix $\mathbf A$, we denote ${\mathbf A}^{(j)}[1, \ \ldots, \ m]$, where $j = 1, \ \ldots, \ n$ and $m = 1, \ \ldots, \ j$, a principal submatrix of the $j$th compound matrix ${\mathbf A}^{(j)}$ which consists of all the minors of the $j$th order of the following form
$$ A\left(\begin{array}{cccccc} 1 & \ldots & m & i_{m+1} & \ldots & i_j \\ 1 & \ldots & m & i_{m+1} & \ldots & i_j \end{array}\right),$$
where $m < i_{m+1} < \ldots < i_j \leq n$.

{\bf Example.} Let ${\mathbf A}$ be a $4 \times 4$ matrix. For $j=3$, we have:
$${\mathbf A}^{(3)}[1] = \left(\begin{array}{ccc} A\left(\begin{array}{ccc} 1 & 2 & 3  \\ 1 & 2 & 3  \end{array}\right) & A\left(\begin{array}{ccc} 1 & 2 & 3  \\ 1 & 2 & 4  \end{array}\right) & A\left(\begin{array}{ccc} 1 & 2 & 3  \\ 1 & 3 & 4  \end{array}\right) \\ A\left(\begin{array}{ccc} 1 & 2 & 4  \\ 1 & 2 & 3  \end{array}\right) & A\left(\begin{array}{ccc} 1 & 2 & 4  \\ 1 & 2 & 4  \end{array}\right) & A\left(\begin{array}{ccc} 1 & 2 & 4  \\ 1 & 3 & 4  \end{array}\right) \\ A\left(\begin{array}{ccc} 1 & 3 & 4  \\ 1 & 2 & 3  \end{array}\right) & A\left(\begin{array}{ccc} 1 & 3 & 4  \\ 1 & 2 & 4  \end{array}\right) & A\left(\begin{array}{ccc} 1 & 3 & 4  \\ 1 & 3 & 4  \end{array}\right) \\ \end{array}\right);$$
$${\mathbf A}^{(3)}[1,2] = \left(\begin{array}{ccc} A\left(\begin{array}{ccc} 1 & 2 & 3  \\ 1 & 2 & 3  \end{array}\right) & A\left(\begin{array}{ccc} 1 & 2 & 3  \\ 1 & 2 & 4  \end{array}\right)  \\ A\left(\begin{array}{ccc} 1 & 2 & 4  \\ 1 & 2 & 3  \end{array}\right) & A\left(\begin{array}{ccc} 1 & 2 & 4  \\ 1 & 2 & 4  \end{array}\right) & \\ \end{array}\right);$$
$${\mathbf A}^{(3)}[1,2,3] =  A\left(\begin{array}{ccc} 1 & 2 & 3  \\ 1 & 2 & 3  \end{array}\right).$$

Now let us prove the following lemma which shows that under some additional conditions on a $Q^2$-matrix $\mathbf A$, some of its diagonal scalings will remain $Q^2$-matrices.

\begin{lemma}\label{main} Let $\mathbf A$ be an $n \times n$ $P$-matrix which satisfies the following conditions:
$${\rm Tr}(({\mathbf A}^{(j)}[1, \ \ldots, \ m])^2)>0, $$
for all $j = 1, \ \ldots, n$ and all $m = 1, \ \ldots, \ j$. Then there is a positive diagonal matrix ${\mathbf D}_{({\mathbf A})}$ with the following two properties:
 \begin{enumerate}
\item[\rm 1.] ${\mathbf D}_{({\mathbf A})}$ is a stabilization matrix for $\mathbf A$.
\item[\rm 2.] ${\mathbf D}_t{\mathbf A} = (t{\mathbf I} + (1-t){\mathbf D}_{({\mathbf A})}){\mathbf A}$ is a $Q^2$-matrix for all $t \in [0,1]$.
\end{enumerate}
\end{lemma}
\begin{proof}
Let us search for the required matrix ${\mathbf D}_{({\mathbf A})}$ in the following form:
\begin{equation}\label{form}
{\mathbf D}_{({\mathbf A})} = \diag\{\epsilon_1, \ \epsilon_2, \ \ldots, \ \epsilon_n\},
\end{equation}
where $1 = \epsilon_1 > \epsilon_2 > \ldots > \epsilon_n > 0$. Since $\mathbf A$ is a $P$-matrix, it satisfies the conditions of Fisher--Fuller theorem (Theorem \ref{FiF}). Thus it is stabilizable. Applying Lemma \ref{diag}, we construct a stabilization matrix $${\mathbf D}_{({\mathbf A})}^0 = \diag\{\epsilon^0_1, \ \epsilon^0_2, \ \ldots, \ \epsilon^0_{n}\},$$ such that any positive diagonal matrix ${\mathbf D}_{({\mathbf A})}$ of Form \ref{form} which satisfies $\epsilon_1 = 1$ and
\begin{equation}\label{*} \frac{\epsilon_i}{\epsilon_{i+1}} \geq \frac{\epsilon_i^0}{\epsilon_{i+1}^0} \qquad i = 1, \ \ldots, \ n-1. \end{equation} is also a stabilization matrix for $\mathbf A$.

Now let us construct the matrix ${\mathbf D}_{({\mathbf A})}$, satisfying both Conditions \ref{*} and the following inequalities:
$${\rm Tr}(({\mathbf D}_{({\mathbf A})})_k^{(j)}{\mathbf A}^{(j)}({\mathbf D}_{({\mathbf A})})_m^{(j)}{\mathbf A}^{(j)}) > 0 $$
for every $j = 1, \ \ldots, \ n$ and every pair $(k,m)$, $1 \leq k,m \leq j$.

We will choose the entries $\epsilon_1, \ \ldots, \ \epsilon_n$ of the required matrix ${\mathbf D}_{({\mathbf A})}$ consequently, taking into account the inequalities $1 = \epsilon_1 > \epsilon_2 > \ldots > \epsilon_n > 0$. First, we put $\epsilon_1 := 1$. Now let us choose $\epsilon_2$. Let us examine all the pairs $(k,m)$ such that $\max(k,m) = 1$. Since $k \geq 1$ and $m \geq 1$, there is only one pair $k = 1$, $m = 1$ which satisfies this condition. So let us choose $\epsilon_2$ such that $\epsilon_2 < \epsilon_2^0$ and ${\rm Tr}(({\mathbf D}_{({\mathbf A})})_1^{(j)}{\mathbf A}^{(j)}({\mathbf D}_{({\mathbf A})})_1^{(j)}{\mathbf A}^{(j)}) > 0 $ for all $j = 1, \ \ldots, \ n-1$. According to Formula \ref{DiagForm}, the following equalities hold for the entries $d^1_{\alpha\alpha}$ of the matrix $({\mathbf D}_{({\mathbf A})})_1^{(j)}$:
\begin{equation}\label{DiagForm1}
d^1_{\alpha\alpha} = \sum_{k = 1}^j\epsilon_{i_k},
\end{equation}
where $\alpha$ is the number in the lexicographic ordering of the set of indices $(i_1, \ \ldots, \ i_j)$, $1 \leq i_1 < \ldots < i_j \leq n$. Examining the lexicographic ordering, it is easy to see that only first ${{n-1} \choose {j-1}}$ of the sets $(i_1, \ \ldots, \ i_j)$ contain the first index $1$ and exactly the first ${{n-1} \choose {j-1}}$ entries $d^1_{\alpha\alpha}$ contain $\epsilon_1 = 1$. Thus we can represent $({\mathbf D}_{({\mathbf A})})_1^{(j)}$ in the following form:
$$({\mathbf D}_{({\mathbf A})})_1^{(j)} = \left(\begin{array}{cc} {\mathbf I}_{{{n-1} \choose {j-1}}}  & 0 \\ 0  & 0 \end{array}\right) + {\mathbf O}_1(\epsilon_2),$$
where ${\mathbf I}_{{{n-1} \choose {j-1}}}$ is an ${{n-1} \choose {j-1}} \times {{n-1} \choose {j-1}}$ identity matrix, ${\mathbf O}_1(\epsilon_2)$ is an ${{n} \choose {j}} \times {{n} \choose {j}}$ positive diagonal matrix. Taking into account Formula \ref{DiagForm1} and the inequalities $1 = \epsilon_1 > \epsilon_2 > \ldots > \epsilon_n$, we obtain the following estimate:
$$\left|{\mathbf O}_1(\epsilon_2)\right| < j\epsilon_2{\mathbf I}. $$

Then, multiplying $\left(\begin{array}{cc} {\mathbf I}_{{{n-1} \choose {j-1}}}  & 0 \\ 0  & 0 \end{array}\right)$ by ${\mathbf A}^{(j)}$, we obtain ${{n-1} \choose {j-1}} \times {{n-1} \choose {j-1}}$ leading principal submatrix of ${\mathbf A}^{(j)}$, which is exactly the matrix ${\mathbf A}^{(j)}[1]$. Thus
$$({\mathbf D}_{({\mathbf A})})_1^{(j)}{\mathbf A}^{(j)}({\mathbf D}_{({\mathbf A})})_1^{(j)}{\mathbf A}^{(j)} = $$ $$\left(\left(\begin{array}{cc} {\mathbf A}^{(j)}[1]  & 0 \\ 0  & 0 \end{array}\right) + {\mathbf O}_1(\epsilon_2){\mathbf A}^{(j)}\right)\left(\left(\begin{array}{cc} {\mathbf A}^{(j)}[1]  & 0 \\ 0  & 0 \end{array}\right) + {\mathbf O}_1(\epsilon_2){\mathbf A}^{(j)}\right) = $$
$$\left(\begin{array}{cc} ({\mathbf A}^{(j)}[1])^2  & 0 \\ 0  & 0 \end{array}\right) + {\mathbf O}_1(\epsilon_2){\mathbf A}^{(j)}\left(\left(\begin{array}{cc} {\mathbf A}^{(j)}[1]  & 0 \\ 0  & 0 \end{array}\right) + {\mathbf O}_1(\epsilon_2){\mathbf A}^{(j)}\right ) + $$ $$ \left(\begin{array}{cc} {\mathbf A}^{(j)}[1]  & 0 \\ 0  & 0 \end{array}\right){\mathbf O}_1(\epsilon_2){\mathbf A}^{(j)} = \left(\begin{array}{cc} ({\mathbf A}^{(j)}[1])^2  & 0 \\ 0  & 0 \end{array}\right) + {\mathbf \Theta}(\epsilon_2).$$

The above equality implies
$${\rm Tr}(({\mathbf D}_{({\mathbf A})})_1^{(j)}{\mathbf A}^{(j)}({\mathbf D}_{({\mathbf A})})_1^{(j)}{\mathbf A}^{(j)}) = {\rm Tr}(({\mathbf A}^{(j)}[1])^2) + {\rm Tr}({\mathbf \Theta}(\epsilon_2)).$$

Since $\left|{\mathbf O}_1(\epsilon_2)\right| < j\epsilon_2{\mathbf I}$, we have the estimates by Formulae \ref{oTrace} and \ref{ooTrace}:
$$\left|{\rm Tr}({\mathbf \Theta}(\epsilon_2))\right| \leq $$ $$ \left|{\rm Tr}({\mathbf O}_1(\epsilon_2)({\mathbf A}^{(j)}[1])^2)\right| + \left|{\rm Tr}({\mathbf O}_1(\epsilon_2){\mathbf A}^{(j)}{\mathbf O}_1(\epsilon_2){\mathbf A}^{(j)})\right| + \left|{\rm Tr}({\mathbf O}_1(\epsilon_2)({\mathbf A}^{(j)}[1])^2)\right| < $$ $$ j\epsilon_2(2{\rm Tr}(|{\mathbf A}^{(j)}[1]|^2) + j\epsilon_2{\rm Tr}(|{\mathbf A}^{(j)}|^2))) \rightarrow 0 \qquad \mbox{as} \ \epsilon_2 \rightarrow 0.$$

Since ${\rm Tr}(({\mathbf A}^{(j)}[1])^2) > 0$, we obtain the estimates
$${\rm Tr}(({\mathbf D}_{({\mathbf A})})_1^{(j)}{\mathbf A}^{(j)}({\mathbf D}_{({\mathbf A})})_1^{(j)}{\mathbf A}^{(j)}) \geq {\rm Tr}(({\mathbf A}^{(j)}[1])^2) - \left|{\rm Tr}({\mathbf \Theta}(\epsilon_2))\right| >$$
$${\rm Tr}(({\mathbf A}^{(j)}[1])^2) - j\epsilon_2(2{\rm Tr}(|{\mathbf A}^{(j)}[1]|^2) + j\epsilon_2{\rm Tr}(|{\mathbf A}^{(j)}|^2))) \geq$$
$$\min_j{\rm Tr}(({\mathbf A}^{(j)}[1])^2) - n\epsilon_2(2\max_j{\rm Tr}(|{\mathbf A}^{(j)}[1]|^2) + n\epsilon_2\max_j{\rm Tr}(|{\mathbf A}^{(j)}|^2))),$$
for all $j = 1, \ \ldots, \ n-1$.

Thus we can choose $\epsilon_2$ such that $\epsilon_2 < \epsilon_2^0$ and
$$\min_j{\rm Tr}(({\mathbf A}^{(j)}[1])^2) - n\epsilon_2(2\max_j{\rm Tr}(|{\mathbf A}^{(j)}[1]|^2) + n\epsilon_2\max_j{\rm Tr}(|{\mathbf A}^{(j)}|^2))) > 0.$$

This inequality implies
$${\rm Tr}(({\mathbf D}_{({\mathbf A})})_1^{(j)}{\mathbf A}^{(j)}({\mathbf D}_{({\mathbf A})})_1^{(j)}{\mathbf A}^{(j)}) > 0$$
for all $j = 1, \ \ldots, \ n-1$.

Assume we have already chosen $\epsilon_1, \ \ldots, \ \epsilon_l$ $(2 \leq l < n)$ such that $1 = \epsilon_1 > \ldots > \epsilon_l > 0$ and
\begin{enumerate}
\item[(a)]$\frac{\epsilon_i}{\epsilon_{i+1}} \geq \frac{\epsilon_i^0}{\epsilon_{i+1}^0}$, $i = 1, \ldots, l-1$.
\item[(b)] ${\rm Tr}(({\mathbf D}_{({\mathbf A})})_k^{(j)}{\mathbf A}^{(j)}({\mathbf D}_{({\mathbf A})})_m^{(j)}{\mathbf A}^{(j)}) > 0$ for all $j = 1, \ \ldots, \ n-1$ and all the pairs $(k,m)$ with $\max(k,m) \leq l-1$, $1 \leq k,m \leq j$.
\end{enumerate}

Let us choose $\epsilon_{l+1}$. For this, we examine all pairs $(k,m)$ with $\max(k,m) = l$ and all $j$ satisfying $j \geq l$. Since $l \leq n-1$, such pairs do exist. Assume that $k = l$, $m \leq l$ (the case of $m = l$, $k \leq l$ is considered analogically). We choose $\epsilon_{l+1}$ satisfying the following conditions:
\begin{enumerate}
\item[(1)] $0 < \epsilon_{l+1} < \epsilon_l$;
\item[(2)] $\frac{\epsilon_l}{\epsilon_{l+1}} \geq \frac{\epsilon_l^0}{\epsilon_{l+1}^0}$;
\item[(3)] ${\rm Tr}(({\mathbf D}_{({\mathbf A})})_l^{(j)}{\mathbf A}^{(j)}({\mathbf D}_{({\mathbf A})})_m^{(j)}{\mathbf A}^{(j)}) > 0$ for all $m$, $1 \leq m \leq l$ and all $j$, $j \geq l$.
\end{enumerate}
For this, we apply the same reasoning as above.

According to Formula \ref{DiagForm}, the following equalities hold for the entries $d^l_{\alpha\alpha}$ of the matrix $({\mathbf D}_{({\mathbf A})})_l^{(j)}$:
$$d^l_{\alpha\alpha} = \sum_{(k_1, \ldots, \ k_l) \subseteq (i_1, \ldots, i_j)}\epsilon_{k_1}\ldots\epsilon_{k_l},$$
where $\alpha$ is the number of the set of indices $(i_1, \ \ldots, i_j)$, $1 \leq i_1 < \ldots < i_j \leq n$, in the lexicographic ordering. Then, considering the lexicographic ordering of the sets $(i_1, \ \ldots, i_j)$, $1 \leq i_1 < \ldots < i_j \leq n$, we obtain that exactly the first ${{n-l}\choose {j-l}}$ of them contains the first indices $(1, \ \ldots, \ l)$. Thus exactly the first ${{n-l}\choose {j-l}}$ entries $d^l_{\alpha\alpha}$ contain the product $\epsilon_1 \epsilon_2 \ldots \epsilon_l$ and we can present $({\mathbf D}_{({\mathbf A})})_l^{(j)}$ in the following form.
$$({\mathbf D}_{({\mathbf A})})_l^{(j)} = \epsilon_1 \ldots \epsilon_l\left(\begin{array}{cc} {\mathbf I}_{{{n-l} \choose {j-l}}}  & 0 \\ 0  & 0 \end{array}\right) + {\mathbf O}_l(\epsilon_{l+1}),$$
where ${\mathbf I}_{{{n-l} \choose {j-l}}}$ is an ${{n-l} \choose {j-l}} \times {{n-l} \choose {j-l}}$ identity matrix, ${\mathbf O}_l(\epsilon_{l+1})$ is an ${{n} \choose {j}} \times {{n} \choose {j}}$ positive diagonal matrix. Since we consider $1 = \epsilon_1 > \epsilon_2 > \ldots > \epsilon_l > \epsilon_{l+1} > \ldots > \epsilon_n > 0$, the following estimate holds:
 $$\left|{\mathbf O}(\epsilon_{l+1})\right| < {j \choose l}\epsilon_1 \ldots \epsilon_{l-1}\epsilon_{l+1}{\mathbf I}.$$

Now let us consider $({\mathbf D}_{({\mathbf A})})_m^{(j)}$, $m \leq l$. Using the same reasoning as above, we represent it in the following form:
 $$({\mathbf D}_{({\mathbf A})})_m^{(j)} = \left(\sum_{(i_1, \ldots, i_m)\subset[l]}\epsilon_{i_1}\ldots\epsilon_{i_m}\right)\left(\begin{array}{cc} {\mathbf I}_{{{n-l} \choose {j-l}}}  & 0 \\ 0  & 0 \end{array}\right) + {\mathbf O}_m(\epsilon_{l+1}),$$
where ${\mathbf I}_{{{n-l} \choose {j-l}}}$ is an ${{n-l} \choose {j-l}} \times {{n-l} \choose {j-l}}$ identity matrix, ${\mathbf O}_m(\epsilon_{l+1})$ is an ${n \choose j} \times {n \choose j}$ positive diagonal matrix.

Since we consider $1 = \epsilon_1 > \epsilon_2 > \ldots > \epsilon_l > \epsilon_{l+1} > \ldots > \epsilon_n > 0$, Formulae \ref{DiagForm} imply the following estimate:
 $$\left|{\mathbf O}_m(\epsilon_{l+1})\right| < {j \choose m} \epsilon_{1} \ldots \epsilon_{m-1}\epsilon_{l+1}{\mathbf I}. $$

Then, multiplying $\left(\begin{array}{cc} {\mathbf I}_{{{n-l} \choose {j-l}}}  & 0 \\ 0  & 0 \end{array}\right)$ by ${\mathbf A}^{(j)}$, we obtain ${{n-l} \choose {j-l}} \times {{n-l} \choose {j-l}}$ leading principal submatrix of ${\mathbf A}^{(j)}$, which is exactly the matrix ${\mathbf A}^{(j)}[1, \ldots, l]$.

Thus we obtain the equality
$$({\mathbf D}_{({\mathbf A})})_l^{(j)}{\mathbf A}^{(j)}({\mathbf D}_{({\mathbf A})})_m^{(j)}{\mathbf A}^{(j)} = $$ $$\left(\epsilon_1 \ldots \epsilon_l\left(\begin{array}{cc} {\mathbf A}^{(j)}[1, \ldots, l]  & 0 \\ 0  & 0 \end{array}\right) + {\mathbf O}_l(\epsilon_{l+1}){\mathbf A}^{(j)}\right)*$$ $$\left(\left(\sum_{(i_1, \ldots, i_m)\subset[l]}\epsilon_{i_1}\ldots\epsilon_{i_m}\right)\left(\begin{array}{cc} {\mathbf A}^{(j)}[1, \ldots, l]   & 0 \\ 0  & 0 \end{array}\right) + {\mathbf O}_m(\epsilon_{l+1}){\mathbf A}^{(j)}\right) = $$
$$\epsilon_1 \ldots \epsilon_l\left(\sum_{(i_1, \ldots, i_m)\subset[l]}\epsilon_{i_1}\ldots\epsilon_{i_m}\right)\left(\begin{array}{cc} ({\mathbf A}^{(j)}[1, \ldots, l])^2  & 0 \\ 0  & 0 \end{array}\right) + $$ $${\mathbf O}_l(\epsilon_{l+1}){\mathbf A}^{(j)}\left(\left(\sum_{(i_1, \ldots, i_m)\subset[l]}\epsilon_{i_1}\ldots\epsilon_{i_m}\right)\left(\begin{array}{cc} {\mathbf A}^{(j)}[1, \ldots, l]   & 0 \\ 0  & 0 \end{array}\right) + {\mathbf O}_m(\epsilon_{l+1}){\mathbf A}^{(j)}\right) + $$ $$\epsilon_1 \ldots \epsilon_l\left(\begin{array}{cc} {\mathbf A}^{(j)}[1, \ldots, l]  & 0 \\ 0  & 0 \end{array}\right){\mathbf O}_m(\epsilon_{l+1}){\mathbf A}^{(j)} =$$
$$ = \epsilon_1 \ldots \epsilon_l\left(\sum_{(i_1, \ldots, i_m)\subset[l]}\epsilon_{i_1}\ldots\epsilon_{i_m}\right)\left(\begin{array}{cc} ({\mathbf A}^{(j)}[1, \ldots, l])^2  & 0 \\ 0  & 0 \end{array}\right) + {\mathbf \Theta}(\epsilon_{l+1}).$$

The above equality implies
$${\rm Tr}(({\mathbf D}_{({\mathbf A})})_l^{(j)}{\mathbf A}^{(j)}({\mathbf D}_{({\mathbf A})})_m^{(j)}{\mathbf A}^{(j)}) = $$
$$\epsilon_1 \ldots \epsilon_l\left(\sum_{(i_1, \ldots, i_m)\subset[l]}\epsilon_{i_1}\ldots\epsilon_{i_m}\right){\rm Tr}(({\mathbf A}^{(j)}[1, \ldots, l])^2) + {\rm Tr}({\mathbf \Theta}(\epsilon_{l+1})).
$$

Since $\left|{\mathbf O}_l(\epsilon_{l+1})\right| < {j \choose l}\epsilon_1 \ldots \epsilon_{l-1}\epsilon_{l+1}{\mathbf I}$
and $\left|{\mathbf O}_m(\epsilon_{l+1})\right| < {j \choose m}\epsilon_1 \ldots \epsilon_{m-1}\epsilon_{l+1}{\mathbf I}$,
 we obtain the following estimate, using Inequalities \ref{oTrace} and \ref{ooTrace}:
$$\left|{\rm Tr}({\mathbf \Theta}(\epsilon_{l+1}))\right| \leq $$
 $$\left(\sum_{(i_1, \ldots, i_m)\subset[l]}\epsilon_{i_1}\ldots\epsilon_{i_m}\right)\left|{\rm Tr}({\mathbf O}_l(\epsilon_{l+1})({\mathbf A}^{(j)}[1,\ \ldots, \ l])^2)\right| + $$ $$\left|{\rm Tr}\left({\mathbf O}_l(\epsilon_{l+1}){\mathbf A}^{(j)}{\mathbf O}_m(\epsilon_{l+1}){\mathbf A}^{(j)}\right)\right| + \epsilon_1 \ldots \epsilon_l\left|{\rm Tr}({\mathbf O}_m(\epsilon_{l+1})({\mathbf A}^{(j)}[1,\ \ldots, \ l])^2)\right|  <$$
$$ {j \choose l}\epsilon_{1} \ldots \epsilon_{l-1}\epsilon_{l+1}\left(\sum_{(i_1, \ldots, i_m)\subset[l]}\epsilon_{i_1}\ldots\epsilon_{i_m}\right){\rm Tr}(|{\mathbf A}^{(j)}[1, \ \ldots, \ l]|^2) + $$ $$ {j \choose l}{j \choose m}\epsilon_{1} \ldots \epsilon_{l-1}\epsilon_{l+1}\epsilon_1 \ldots \epsilon_{m-1}\epsilon_{l+1}{\rm Tr}(|{\mathbf A}^{(j)}|^2) + $$ $$ {j \choose m}\epsilon_{1} \ldots \epsilon_{m-1}\epsilon_{l+1}\epsilon_1 \ldots \epsilon_l{\rm Tr}(|{\mathbf A}^{(j)}[1, \ \ldots, \ l]|^2) \rightarrow 0 \qquad \mbox{as} \ \epsilon_{l+1} \rightarrow 0.$$
Since ${\rm Tr}(({\mathbf A}^{(j)}[1, \ldots, l])^2) > 0$, we have
$${\rm Tr}(({\mathbf D}_{({\mathbf A})})_l^{(j)}{\mathbf A}^{(j)}({\mathbf D}_{({\mathbf A})})_m^{(j)}{\mathbf A}^{(j)}) \geq $$ $$ \epsilon_1 \ldots \epsilon_l\left(\sum_{(i_1, \ldots, i_m)\subset[l]}\epsilon_{i_1}\ldots\epsilon_{i_m}\right){\rm Tr}(({\mathbf A}^{(j)}[1, \ \ldots, \ l])^2) - \left|{\rm Tr}({\mathbf \Theta}(\epsilon_{l+1}))\right| >$$ $$ \epsilon_1 \ldots \epsilon_l\left(\sum_{(i_1, \ldots, i_m)\subset[l]}\epsilon_{i_1}\ldots\epsilon_{i_m}\right){\rm Tr}(({\mathbf A}^{(j)}[1, \ \ldots, \ l])^2) - $$ $$ {j \choose l}\epsilon_{1} \ldots \epsilon_{l-1}\epsilon_{l+1}\left(\sum_{(i_1, \ldots, i_m)\subset[l]}\epsilon_{i_1}\ldots\epsilon_{i_m}\right){\rm Tr}(|{\mathbf A}^{(j)}[1, \ \ldots, \ l]|^2) - $$ $$ {j \choose l}{j \choose m}\epsilon_{1} \ldots \epsilon_{l-1}\epsilon_{l+1}\epsilon_1 \ldots \epsilon_{m-1}\epsilon_{l+1}{\rm Tr}(|{\mathbf A}^{(j)}|^2) -$$ $${j \choose m}\epsilon_{1} \ldots \epsilon_{m-1}\epsilon_{l+1}\epsilon_1 \ldots \epsilon_l{\rm Tr}(|{\mathbf A}^{(j)}[1, \ \ldots, \ l]|^2).$$

Then we can choose $\epsilon_{l+1}$ such that $\epsilon_{l+1} \leq \frac{\epsilon_{l+1}^0}{\epsilon_l^0}\epsilon_l$ and
$$ \epsilon_1 \ldots \epsilon_l\left(\sum_{(i_1, \ldots, i_m)\subset[l]}\epsilon_{i_1}\ldots\epsilon_{i_m}\right)\min_j{\rm Tr}(({\mathbf A}^{(j)}[1, \ \ldots, \ l])^2) - $$ $$ {n \choose l}\epsilon_{1} \ldots \epsilon_{l-1}\epsilon_{l+1}\left(\sum_{(i_1, \ldots, i_m)\subset[l]}\epsilon_{i_1}\ldots\epsilon_{i_m}\right)\max_j{\rm Tr}(|{\mathbf A}^{(j)}[1, \ \ldots, \ l]|^2) - $$ $$ {n \choose l}{j \choose m}\epsilon_{1} \ldots \epsilon_{l-1}\epsilon_{l+1}\epsilon_1 \ldots \epsilon_{m-1}\epsilon_{l+1}\max_j{\rm Tr}(|{\mathbf A}^{(j)}|^2) -$$ $$ {n \choose m}\epsilon_{1} \ldots \epsilon_{m-1}\epsilon_{l+1}\epsilon_1 \ldots \epsilon_l\max_j{\rm Tr}(|{\mathbf A}^{(j)}[1, \ \ldots, \ l]|^2) > 0.$$
This inequality implies ${\rm Tr}(({\mathbf D}_{({\mathbf A})})_l^{(j)}{\mathbf A}^{(j)}({\mathbf D}_{({\mathbf A})})_m^{(j)}{\mathbf A}^{(j)}) > 0$ for all $j = l, \ \ldots, \ n-1$.

Applying the above reasoning $n-1$ times, we construct all the entries $\epsilon_1, \ \ldots, \ \epsilon_n$ of the matrix ${\mathbf D}_{({\mathbf A})}$. For $j = n$, the inequality ${\rm Tr}(({\mathbf D}_{({\mathbf A})})_k^{(n)}{\mathbf A}^{(n)}({\mathbf D}_{({\mathbf A})})_m^{(n)}{\mathbf A}^{(n)}) > 0$ is obvious
for any positive diagonal matrix ${\mathbf D}_{({\mathbf A})}$ and any pair $(k,m)$, $1 \leq k,m \leq n$.

Since the matrix ${\mathbf D}_{({\mathbf A})}$ satisfies Conditions \ref{*}, we obtain by Lemma \ref{diag}, that it is a stabilization matrix for $\mathbf A$. By the construction, it satisfies the conditions of Lemma \ref{tech}. Applying Lemma \ref{tech}, we complete the proof.
\end{proof}

Now let us recall the following basic definitions and facts, namely, Sylvester's Determinant Identity (see, for example, \cite{PINK}, p. 3) and the Schur complement formula (see, for example, \cite{CARLMAR}, \cite{CRAH}). For a given $n \times n$ matrix $\mathbf A$, two sets of indices $(i_1, \ \ldots, i_k)$ and $(j_1, \ \ldots, j_k)$ from $[n]$ and each $l \in [n]\setminus (i_1, \ \ldots, i_k)$, $r \in [n]\setminus (j_1, \ \ldots, j_k)$, we define $$b_{lr} = A \left(\begin{array}{cccc}i_1 & \ldots & i_k & l \\ j_1 & \ldots & j_k & r \end{array}\right).$$ (In such notations, we always mean that the set of indices is arranged in natural increasing order). Then the following identity (Sylvester's Determinant Identity) holds for the minors of the $(n - k)\times(n-k)$ matrix ${\mathbf B} = \{b_{lr}\}$:
$$ B \left(\begin{array}{ccc}l_1 & \ldots & l_p \\ r_1 & \ldots & r_p \end{array}\right) = $$
$$ A \left(\begin{array}{ccc}i_1 & \ldots & i_k \\ j_1 & \ldots & j_k \end{array}\right)^{p-1} A \left(\begin{array}{cccccc}i_1 & \ldots & i_k & l_1 & \ldots & l_p \\ j_1 & \ldots & j_k & r_1 & \ldots & r_p \end{array}\right). $$

Now let us write an initial $n \times n$ matrix $\mathbf A$ in the following block form:
          $${\mathbf A} = \left(\begin{array}{ccc} {\mathbf A}_{kk} & {\mathbf A}_{k,n-k} \\ {\mathbf A}_{n-k,k} & {\mathbf A}_{n-k,n-k} \end{array}\right),$$
where ${\mathbf A}_{kk}$ is a nonsingular leading principal submatrix of ${\mathbf A}$, spanned by the basic vectors $e_1, \ \ldots, \ e_k$, $1 \leq k < n$, ${\mathbf A}_{n-k,n-k}$ is the principal submatrix spanned by the remaining basic vectors. Then the Schur complement $A|A_{kk}$ of ${\mathbf A}_{kk}$ in $\mathbf A$ is defined as follows:
 $$A|A_{kk} = {\mathbf A}_{n-k,n-k}-{\mathbf A}_{n-k,k}{\mathbf A}_{kk}^{-1}{\mathbf A}_{k,n-k}.$$

For the entries of the Schur complement, we have the following equality: $A|A_{kk} = \{c_{lr}\}$, where
$$ c_{lr} = A\left(\begin{array}{cccc} 1 & \ldots & k & l \\ 1 & \ldots & k & r \end{array}\right)/A\left(\begin{array}{ccc} 1 & \ldots & k \\ 1 & \ldots & k \end{array}\right), $$
 $n-k \leq l,r \leq n$.

The following formula connects the inverse of the Schur complement $A|A_{kk}$ with a principal submatrix of ${\mathbf A}^{-1}$:
\begin{equation}\label{ShurInv} (A|A_{kk})^{-1} = {{\mathbf A}^{-1}}_{n-k,n-k}.\end{equation}

Using the properties of Schur complements, we'll prove the following statement.
\begin{lemma}\label{Shur} Let a $P$-matrix $\mathbf A$ be a $Q^2$-matrix. Let ${\mathbf A}$ have at least one nested sequence of principal submatrices, each of which is a $Q^2$-matrix. Then there is a permutation $\theta$ of $[n]$ with the permutation matrix ${\mathbf P}_{\theta}$ such that the matrix ${\mathbf B} = {\mathbf P}_{\theta}^{-1}{\mathbf A}^{-1}{\mathbf P}_{\theta}$ also is a $P$-matrix, a $Q^2$-matrix and satisfies the following conditions:
 \begin{equation}\label{TRcond}{\rm Tr}(({\mathbf B}^{(j)}[1, \ \ldots, \ m])^2) > 0;\end{equation}
for all $j = 1, \ \ldots, \ n$ and all $m = 1, \ \ldots, \ j$.
\end{lemma}
\begin{proof} Let us use the notation $A\left[\begin{array}{ccc}i_1 & \ldots & i_j \\ i_1 & \ldots & i_j \end{array}\right]$ for the $j \times j$ principal submatrix of $\mathbf A$ determined by the rows with the indices $(i_1, \ \ldots, \ i_j)$ and the columns with the same indices. (Note that the indices $(i_1, \ \ldots, \ i_j)$ in this notation are re-ordered in the increasing order, as they are in the initial matrix.) We consider the given nested sequence of the principal submatrices of $\mathbf A$, each of which is a $Q^2$-matrix:
$$a_{i_1 i_1}, \  A\left[\begin{array}{cc}i_1 & i_2 \\ i_1 & i_2 \end{array}\right], \ \ldots, \ A\left[\begin{array}{ccc}i_1 & \ldots & i_n \\ i_1 & \ldots & i_n \end{array}\right].$$
This sequence is defined by the permutation $\tau = (i_1, \ \ldots, \ i_n)$ of the set of indices $[n]$. Let us construct the permutation $\theta$ by the following rule:
$\theta(i_m) = n-m+1$, $m = 1, \ \ldots, \ n$. In this case, we obtain the following equalities for the principal submatrices of the matrix $\widetilde{{\mathbf A}} = {\mathbf P}_{\theta}^{-1}{\mathbf A}{\mathbf P}_{\theta}$:
$$A\left[\begin{array}{ccc}i_1 & \ldots & i_{n-m} \\ i_1 & \ldots & i_{n-m} \end{array}\right] = \widetilde{A}\left[\begin{array}{ccc}m+1 & \ldots & n \\ m+1 & \ldots & n \end{array}\right], \qquad m = 0, \ \ldots, \ n-1.$$
Now let us consider the matrix ${\mathbf B} = (\widetilde{{\mathbf A}})^{-1} = {\mathbf P}_{\theta}^{-1}{\mathbf A}^{-1}{\mathbf P}_{\theta}$ and its Schur complements of the form $B|B_{mm}$, where ${\mathbf B}_{mm}$ is the principal submatrix of $\mathbf B$ which consists of the first $m$ rows and columns, $m = 1, \ \ldots, \ n$. Using Formula \ref{ShurInv}, we obtain:
$$(B|B_{mm})^{-1} = \widetilde{A}\left[\begin{array}{ccc}m+1 & \ldots & n \\ m+1 & \ldots & n \end{array}\right] =
A\left[\begin{array}{ccc}i_1 & \ldots & i_{n-m} \\ i_1 & \ldots & i_{n-m} \end{array}\right], \qquad m = 1, \ \ldots, \ n-1.$$

By the conditions, $A\left[\begin{array}{ccc}i_1 & \ldots & i_{n-m} \\ i_1 & \ldots & i_{n-m} \end{array}\right]$ is a $Q^2$-matrix, and it is a $P$-matrix since it is a principal submatrix of a $P$-matrix. Thus $(B|B_{mm})^{-1}$ is also a $Q^2$-matrix and a $P$-matrix.

By Lemma \ref{P^2} and Lemma \ref{P} we obtain that ${\mathbf B}$ is also a $Q^2$-matrix and a $P$-matrix. Now we'll prove Conditions \ref{TRcond}. Applying Lemma \ref{P^2} and Lemma \ref{P} to all $(B|B_{mm})^{-1}$, $m = 1, \ \ldots, \ n-1$, each of which is a $P$-matrix and a $Q^2$-matrix, we obtain that $B|B_{mm}$ is also a $P$-matrix and a $Q^2$-matrix for every $m = 1, \ \ldots, \ n-1$.

Now let us fix $m$, $1 \leq m \leq n-1$ and consider the matrices ${\mathbf B}^{(j)}[1, \ \ldots, \ m]$ for each positive integer $j$, $m < j \leq n$. The matrix ${\mathbf B}^{(j)}[1, \ \ldots, \ m]$ consists of all the minors of the form
 $$B\left(\begin{array}{cccccc}1 & \ldots & m & i_{m+1} & \ldots & i_j \\ 1 & \ldots & m & k_{m+1} & \ldots & k_j \end{array}\right), $$
where $m+1 \leq i_{m+1} < \ldots < i_j \leq n$, $m+1 \leq k_{m+1} < \ldots < k_j \leq n$.
Applying the Silvester determinant identity, we obtain
$$(B|B_{mm})\left(\begin{array}{ccc} l_{1} & \ldots & l_{j-m} \\
r_{1} & \ldots & r_{j-m} \end{array}\right)B\left(\begin{array}{cccccc}1 & \ldots & m  \\ 1 & \ldots & m  \end{array}\right)^{j-m} = $$ $$ B\left(\begin{array}{cccccc}1 & \ldots & m  \\ 1 & \ldots & m \end{array}\right)^{j-m-1}B\left(\begin{array}{cccccc}1 & \ldots & m & l_{1} & \ldots & l_{j-m} \\ 1 & \ldots & m & r_{1} & \ldots & r_{j-m} \end{array}\right).$$
Thus
$$(B|B_{mm})\left(\begin{array}{ccc} l_{1} & \ldots & l_{j-m} \\
r_{1} & \ldots & r_{j-m} \end{array}\right) = $$ $$ \left.B\left(\begin{array}{cccccc}1 & \ldots & m & l_{1} & \ldots & l_{j-m} \\ 1 & \ldots & m & r_{1} & \ldots & r_{j-m} \end{array}\right)\right/B\left(\begin{array}{cccccc}1 & \ldots & m  \\ 1 & \ldots & m \end{array}\right).$$
And we obtain the following equality for the compound matrices:
\begin{equation}\label{Comp}(B|B_{mm})^{(j-m)} = \left.{\mathbf B}^{(j)}[1, \ \ldots, \ m]\right/B\left(\begin{array}{cccccc}1 & \ldots & m  \\ 1 & \ldots & m \end{array}\right).\end{equation}
Applying the Cauchy--Binet formula,
\begin{equation}\label{Comp2}
((B|B_{mm})^2)^{(j-m)} = ((B|B_{mm})^{(j-m)})^2
=  \end{equation} $$ \left.({\mathbf B}^{(j)}[1, \ \ldots, \ m])^2\right/B\left(\begin{array}{cccccc}1 & \ldots & m  \\ 1 & \ldots & m \end{array}\right)^2, $$
for all $m = 1, \ \ldots, \ n$ and all $j$, $j = m+1, \ \ldots, \ n$.
Since $B|B_{mm}$ is a $Q^2$-matrix we have that ${\rm Tr}(((B|B_{mm})^2)^{(j-m)}) > 0$ for all $j = m+1, \ \ldots, \ n$. Since $\mathbf B$ is a $P$-matrix, we have $B\left(\begin{array}{cccccc}1 & \ldots & m  \\ 1 & \ldots & m \end{array}\right) > 0$ for all $m = 1, \ \ldots, \ n$. Thus Equality \ref{Comp2} shows that
$${\rm Tr}(({\mathbf B}^{(j)}[1, \ \ldots, \ m])^2)
= B\left(\begin{array}{cccccc}1 & \ldots & m  \\ 1 & \ldots & m \end{array}\right)^{2}{\rm Tr}(((B|B_{mm})^2)^{(j-m)}) > 0. $$
\end{proof}

The proof of Carlson's theorem uses the fact that if $\mathbf A$ is a sign-symmetric $P$-matrix and ${\mathbf D}$ is an arbitrary positive diagonal matrix then the matrix ${\mathbf D}{\mathbf A}$ is a $P^2$-matrix. However,
the proof of our main result does not require $({\mathbf D}{\mathbf A})^2$ to be a $P$-matrix (and even a $Q$-matrix) for {\bf every} positive diagonal matrix ${\mathbf D}$. This condition is too strong. In the proof below, we use a much weaker condition: under certain assumptions on a $Q^2$-matrix $\mathbf A$, there is {\bf at least one} stabilization matrix ${\mathbf D}_{({\mathbf A})}$ such that $(t{\mathbf I} + (1-t){\mathbf D}_{({\mathbf A})}){\mathbf A}$ is a $Q^2$-matrix for every $t \in [0,1]$ (note that here we multiply $\mathbf A$ by a positive diagonal matrix of a very special form).

\begin{pmain}
Applying Lemma \ref{Shur} to ${\mathbf A}$, we find a permutation matrix ${\mathbf P}_{\theta}$ and construct a $P$-matrix ${\mathbf B} = {\mathbf P}_{\theta}^{-1}{\mathbf A}^{-1}{\mathbf P}_{\theta}$, which is also a $Q^2$-matrix and satisfies the following conditions:
$${\rm Tr}(({\mathbf B}^{(j)}[1, \ \ldots, \ m])^2)> 0, $$
for all $m = 1, \ \ldots, \ j$ and all $j = 1, \ \ldots, \ n$.

Since $\mathbf B$ satisfies the conditions of Lemma \ref{main}, it is stabilizable and we construct a stabilization matrix ${\mathbf D}_{({\mathbf B})}$ such that ${\mathbf D}_t{\mathbf B} = (t{\mathbf I} + (1-t){\mathbf D}_{({\mathbf B})}){\mathbf B}$ is a $Q^2$-matrix for all $t \in [0,1]$.

Now we apply to $\mathbf B$ the reasoning of the proof of Carlson's theorem (see \cite{CARL2}). Since ${\mathbf D}_{({\mathbf B})}$ is a stabilization matrix for $\mathbf B$, we obtain that the matrix ${\mathbf D}_{({\mathbf B})}{\mathbf B}$ has a positive distinct spectrum. Let us consider $({\mathbf D}_t{\mathbf B})^2$ which is a $Q$-matrix for any $t \in [0,1]$ (as it was shown above).
Then, applying Corollary \ref{hersh} (Hershkowitz), we obtain that $({\mathbf D}_t{\mathbf B})^2$ cannot have nonpositive real eigenvalues. Since the eigenvalues of $({\mathbf D}_t{\mathbf B})^2$ are just the squares of the eigenvalues of ${\mathbf D}_t{\mathbf B}$, we obtain that ${\mathbf D}_t{\mathbf B}$ can not have eigenvalues on the imaginary axis, for any $t \in [0,1]$. For $t = 0$, we have that ${\mathbf D}_t{\mathbf B} = {\mathbf D}_{({\mathbf B})}{\mathbf B}$ and all of its eigenvalues are on the right-hand side of the complex plane (they are all positive). Since the eigenvalues changes continuously on $t$, we have that the eigenvalues of ${\mathbf D}_t{\mathbf B}$ can not cross the imaginary axis for any $t \in [0,1]$. So they must stay in the right-hand side of the complex plane. Thus putting $t = 0$, we obtain that ${\mathbf B}$ is positively stable. Taking into account that $ {\mathbf B} = {\mathbf P}_{\theta}^{-1}{\mathbf A}^{-1}{\mathbf P}_{\theta}$ and the eigenvalues of $\mathbf B$ are just the inverses of the eigenvalues of ${\mathbf A}$, we conclude that ${\mathbf A}$ is also positively stable.\end{pmain}

Note, that applying Corollary \ref{hersh} to the matrices $({\mathbf D}_t{\mathbf B})^2$, $t \in [0,1]$ we can localize the spectrum of the initial matrix $\mathbf A$ more precisely: $\lambda \in \sigma({\mathbf A})$ implies $$|\arg(\lambda)| < \frac{\pi}{2} - \frac{\pi}{2n}.$$

\section{Examples and applications}
Theorem \ref{maintheorem} implies Theorem \ref{Carl} (Carlson) and Theorem \ref{Tang} (Tang et al).

\begin{pcarl}If $\mathbf A$ is a sign-symmetric $P$-matrix then all its principal submatrices are also sign-symmetric $P$-matrices. It is easy to see that ${\mathbf A}$ is a $Q^2$-matrix and all its principal submatrices are also $Q^2$-matrices. Thus we obtain that ${\mathbf A}$ satisfies the conditions of Theorem \ref{maintheorem}. Applying Theorem \ref{maintheorem} we get that $\mathbf A$ is positively stable.
\end{pcarl}

\begin{ptang} For the proof, it is enough to show that ${\mathbf A}^2$ is a $Q$-matrix and $\mathbf A$ has a nested sequence of principal submatrices each of which is also a $Q^2$-matrix. In this case, we can show even more: that every principal submatrix of $\mathbf A$ is a $Q^2$-matrix.

First, let us show that if $\mathbf A$ is strictly row square diagonally dominant for every order of minors then so is every principal submatrix of $\mathbf A$. (The same is true for the column dominance.) In this case, we have, by Definition 1.10:
$$\left(A\left(\begin{array}{c} \alpha \\
\alpha \end{array}\right)\right)^2 > \sum_{\alpha,\beta \subset [n], \alpha \neq \beta}\left( A\left(\begin{array}{c} \alpha \\
\beta \end{array}\right)\right)^2,$$
 where $\alpha = (i_1, \ \ldots, \ i_j)$, $1 \leq i_1 < \ \ldots < i_j \leq n$, $\beta = (l_1, \ \ldots, \ l_j)$, $1 \leq l_1 < \ \ldots < l_j \leq n$, $j = 1, \ \ldots, \ n$.
Let us fix an arbitrary $(n-1) \times (n-1)$ submatrix ${\mathbf A}_k$ obtained from $\mathbf A$ by deleting the row and the column with the number $k$, $1 \leq k \leq n$.
Since all the terms in the right-hand side of the above inequality are positive, we obtain for $j = 1, \ \ldots, \ n-1$:
$$\sum_{\alpha,\beta \subset [n], \alpha \neq \beta}\left( A\left(\begin{array}{c} \alpha \\
\beta \end{array}\right)\right)^2> \sum_{\alpha,\beta \subset [n]\setminus\{k\}, \alpha \neq \beta}\left( A\left(\begin{array}{c} \alpha \\
\beta \end{array}\right)\right)^2.$$

Thus ${\mathbf A}_k$ is strictly row square diagonally dominant for every order of minors $1, \ \ldots, \ n-1$ and for each value of $k$, $1 \leq k \leq n$.

Now let us show that both ${\mathbf A}^2$ and ${\mathbf A}_k^2$ are $Q$-matrices. The proof copies the corresponding reasoning of Tang et al (see \cite{TSO}, p. 27, proof of Theorem 3). First, we obtain the estimate:
 $$A^2\left(\begin{array}{c} \alpha \\
\alpha \end{array}\right) = \sum_{\alpha,\beta \subset [n]} A\left(\begin{array}{c} \alpha \\
\beta \end{array}\right)A\left(\begin{array}{c} \beta \\
\alpha \end{array}\right) = $$ $$ \left(A\left(\begin{array}{c} \alpha \\
\alpha \end{array}\right)\right)^2 + \sum_{\alpha,\beta \subset [n], \ \alpha \neq \beta} A\left(\begin{array}{c} \alpha \\
\beta \end{array}\right)A\left(\begin{array}{c} \beta \\
\alpha \end{array}\right) >$$
$$[\mbox{by the condition of square diagonal dominance}] $$
$$> \sum_{\alpha,\beta \subset [n], \ \alpha \neq \beta}\left(\left( A\left(\begin{array}{c} \alpha \\
\beta \end{array}\right)\right)^2 + A\left(\begin{array}{c} \alpha \\
\beta \end{array}\right)A\left(\begin{array}{c} \beta \\
\alpha \end{array}\right) \right).$$

It follows that:
$$\sum_{|\alpha| = j}A^2\left(\begin{array}{c} \alpha \\
\alpha \end{array}\right) > \sum_{|\alpha| = j}\left(\sum_{\alpha,\beta \subset [n], \ \alpha \neq \beta}\left(\left( A\left(\begin{array}{c} \alpha \\
\beta \end{array}\right)\right)^2 + A\left(\begin{array}{c} \alpha \\
\beta \end{array}\right)A\left(\begin{array}{c} \beta \\
\alpha \end{array}\right) \right)\right) = $$ $$ = \sum_{|\alpha| = j}\left(\frac{1}{2}\sum_{\alpha,\beta \subset [n], \ \alpha \neq \beta}\left(\left( A\left(\begin{array}{c} \alpha \\
\beta \end{array}\right)\right)^2 + 2A\left(\begin{array}{c} \alpha \\
\beta \end{array}\right)A\left(\begin{array}{c} \beta \\
\alpha \end{array}\right) + \left( A\left(\begin{array}{c} \beta \\
\alpha \end{array}\right)\right)^2 \right) \right) = $$ $$ = \sum_{|\alpha| = j}\frac{1}{2}\left(\sum_{\alpha,\beta \subset [n], \ \alpha \neq \beta}\left( A\left(\begin{array}{c} \alpha \\
\beta \end{array}\right) + A\left(\begin{array}{c} \beta \\
\alpha \end{array}\right)\right)^2 \right) > 0.$$

Thus a strictly row square diagonally dominant $P$-matrix $\mathbf A$ is a $Q^2$-matrix as well as all its principal submatrices. Applying Theorem \ref{maintheorem}, we obtain that $\mathbf A$ is positively stable.
\end{ptang}

The following matrix classes obviously satisfy the conditions of Theorem \ref{maintheorem}.
\begin{enumerate}
\item[\rm 1.] Hermitian positive definite matrices.
\item[\rm 2.] Strictly totally positive matrices (an $n \times n$ matrix ${\mathbf A}$ is called {\it strictly totally positive} if $\mathbf A$ is (entry-wise) positive and its $j$th compound matrices ${\mathbf A}^{(j)}$ are also positive for all $j = 2, \ \ldots, \ n$).
\item[\rm 3.] Strictly totally J-sign-symmetric matrices (for the definition and properties see \cite{KU}).
\end{enumerate}
\section*{Acknowledgments}
The research leading to these results was carried out at Technische Universit\"{a}t Berlin and has received funding from
the European Research Council under the European Unions Seventh Framework Programme (FP7/20072013)/ERC grant
agreement n 259173.
The author thanks Dr. Michal Wojtylak and Alexander Dyachenko for their valuable comments.

\end{document}